\documentclass[reqno,twoside,11pt]{amsart}

\usepackage{amsmath,amsfonts,calrsfs,fullpage,amssymb,color,verbatim,eucal,yfonts,mathrsfs}

\newtheorem{Theorem}{Theorem}[section]
\newtheorem{Definition}{Definition}
\newtheorem{Proposition}[Theorem]{Proposition}
\newtheorem{Lemma}[Theorem]{Lemma}

\newtheorem{Remark}[Theorem]{Remark}

\newtheorem{Hypothesis}{Hypothesis}
\makeatletter
\@addtoreset{equation}{section}

\makeatother

\def\a{\alpha}
\def\le{\left}

\def\la{\lambda}

\def\ds{\displaystyle}

\def\e{\epsilon}

\title{\bf Nonlinear random perturbations of PDEs and quasi-linear equations in Hilbert spaces depending on a small parameter}\date{}

\author[S. Cerrai]{Sandra Cerrai}
\address{Department of Mathematics\\
University of Maryland\\ 
4176 Campus Drive, College Park, MD 20742, United States}
\email{cerrai@umd.edu}

\author[G. Guatteri]{Giuseppina Guatteri}
\address{Dipartimento di Matematica F. Brioschi\\
Politecnico di Milano\\
via Bonardi 9, 20133 Milano, Italy}
\email{giuseppina.guatteri@polimi.it}

\author[G. Tessitore]{Gianmario Tessitore}
\address{
Dipartimento di Matematica e Applicazioni\\
Universit\`a di Milano Bicocca\\
Via Roberto Cozzi 55, 20125 Milano, Italy}
\email{gianmario.tessitore@unimib.it}

\subjclass[2010]{}


\begin{document}

 \begin{abstract}  We study a class of quasi-linear parabolic equations defined on a separable Hilbert space, depending on a small parameter in front of the second order term. Through the nonlinear semigroup associated with such equation, we introduce the corresponding SPDE and we study the asymptotic behavior of its solutions, depending on the small parameter. We show that a large deviations principle holds and we give an explicit description of the action functional.

 \end{abstract}

 \maketitle

\section{Introduction}
\label{sec1}

Consider the partial differential equation
\begin{equation}
	\label{intro1}
	\frac{dX^x}{dt}(t)=AX^x(t)+b(X^x(t)),\ \ \ \ \ X^x(0)=x \in\,H,
\end{equation}
defined on a separable Hilbert space $H$, endowed with the scalar product $\langle\cdot,\cdot\rangle_H$ and the corresponding norm $\Vert\cdot\Vert_H$. Here $A:D(A)\subset H\to H$ is the generator of a strongly continuous semigroup and $b:D(b)\subseteq H\to H$ is some non-linear mapping.
Next, consider the following stochastic perturbation of \eqref{intro1}
\begin{equation}
\label{intro2}
dX^x_\e(t)=\left[AX^x_\e(t)+b(X^x_\e(t))\right]\,dt+\sqrt{\e}\,\sigma(X^x_\e(t))\,dW_t,\ \ \ \ \ X^x_\e(0)=x \in\,H,	
\end{equation}
where $\e>0$ is a small parameter, $W_t$, $t\geq 0$, is a cylindrical Wiener process and $\sigma$ is a mapping, defined on $H$ and taking values in some space of bounded linear operators defined on the reproducing kernel of the noise into  $H$. We assume that the differential operator $A$, the coefficients $b$ and $\sigma$ and the noise $W_t$ are such that both \eqref{intro1} and \eqref{intro2} are well-posed.

If the parameter $\e$ is small,  the trajectories of the perturbed system \eqref{intro2} remain close to those of the unperturbed system \eqref{intro1} on any bounded time interval. In particular, if there exist a domain $G\subset H$ and a point $x_0 \in\,G$ such that any trajectory of \eqref{intro1} starting in  $G$ remains in $G$ and converges to $x_0$, as time goes to infinity,  then with overwhelming probability the trajectories of \eqref{intro2} starting from any $x \in G$ enter any neighborhood of $x_0$, before eventually leaving the domain $G$ because of the effect of the noise. As know, this is a consequence of the large deviations of $X_\e(t)$ from $X(t)$
which are described by the {\em action functional}
\[I^x_T(f)=\frac 12 \inf\left\{\int_0^T \Vert \varphi(t)\Vert_H^2\,dt\,:\,f=X^{x,\varphi}\right\},\]
where we have denoted by $X^{x,\varphi}$   the solution of the controlled version of \eqref{intro1}
\[\frac{dX^{x,\varphi}}{dt}(t)=AX^{x,\varphi}(t)+b(X^{x,\varphi}(t))+\sigma(X^{x,\varphi}(t))\varphi(t),\ \ \ \ \ X^{x,\varphi}(0)=x,\]
and by the {\em quasi-potential}
\[V(x_0,x)=\inf\left\{I_T(f)\,:\,f \in\,C([0,T];H),\ f(0)=x_0,\ f(T)=x,\ T>0\right\}.\]

It is known that the stochastic PDE \eqref{intro2} is related to the linear  Kolmogorov equation on the Hilbert space $H$ 
\begin{equation}
\label{intro5}
\left\{\begin{array}{l}
\ds{D_tu_\e(t,x)=\frac \e2 \text{Tr} \left[\sigma\sigma^\star(x)D^2_x	 u_\epsilon(t,x)\right]+\langle Ax+b(x),D u_\epsilon(t,x)\rangle_H,\ \ \ x \in\,H,\ \ \ t>0,}\\[10pt]
\ds{u_\epsilon(0,x)=g(x),\ \ \ x \in\,H.}
	\end{array}\right.
\end{equation}
Actually, under suitable conditions on the operator $A$, the coefficients $b$ and $\sigma$ and the initial condition $g$,  
equation \eqref{intro5} admits a unique classical solution $u_\e$, which can be written in terms of the linear transition semigroup $P^\e_t$ associated with \eqref{intro2}. Namely
\[u_\e(t,x)=P^\e_tg(x)=\mathbb{E}g(X_\e(t,x)),\ \ \ \ t\geq 0,\ \ x \in\,H.\]
In particular, the description of the small noise asymptotics of the solutions of equation \eqref{intro2} provided by the theory of large deviations  allows to give a detailed description of the long-time behavior of the solutions of infinite dimensional PDE \eqref{intro5}.

\medskip

In \cite{FK}, Freidlin and Koralov have considered more general stochastic  perturbations of the dynamical system \eqref{intro1}, when $H=\mathbb{R}^d$, $A=0$ and  $b:\mathbb{R}^d\to \mathbb{R}^d$ is a Lipschitz-continuous mapping. They have introduced the following quasi-linear parabolic problem  
\begin{equation}
\label{intro6}
\left\{\begin{array}{l}
\ds{\partial_tu_\e(t,x)=\frac \e2 \sum_{i, j=1}^d a_{i,j}(x,u_\e(t,x))\,\partial_{ij}	 u_\epsilon(t,x) +\sum_{i=1}^d b_i(x)\,\partial_iu_\epsilon(t,x),\ \ \ x \in\,\mathbb{R}^d,\ \ \ t>0,}\\[10pt]
\ds{u_\epsilon(0,x)=g(x),\ \ \ x \in\,\mathbb{R}^d,}
	\end{array}\right.
\end{equation}
where $a_{ij}(x,r)=(\sigma\sigma^\star)_{ij}(x,r)$, and by invoking the classical theory of quasi-linear PDEs,  they have shown that, under reasonable assumptions on the coefficients $f$ and $\sigma$,  equation \eqref{intro6} admits a unique classical solution $u_\e$. Next, for every $t>0$ and $x \in\,\mathbb{R}^d$, they have  introduced the following randomly perturbed system
\begin{equation}
\label{intro7}
\left\{\begin{array}{l}
\ds{dX_\e^{t,x}(s)=b(X_\e^{t,x}(s))\,ds+\sqrt{\e}\,\sigma(X_\e^{t,x}(s),u_\e(t-s,X_\e^{t,x}(s)))\,dB_s,}\\[10pt]
\ds{X_\e^{t,x}(0)=x,}
	\end{array}\right.
\end{equation}
where $B_t$, $t\geq 0$, is a $d$-dimensional Brownian motion. As in the linear case, the PDE \eqref{intro6} and the SDE \eqref{intro7}  are related by the following relation
\begin{equation}
\label{maximum}
u_\e(t,x)=\mathbb{E}g(X^{t,x}_\e(t))=:T^\e_tg(x),	
\end{equation}
but now  $T^\e_t$ is a non-linear semigroup. This is in fact  reason why equation \eqref{intro7} can be seen as a {\em non-linear} perturbation of the deterministic system.

The study of  the large deviation principle and of  the quasi-potential for \eqref{intro7}, has allowed Freidlin and Koralov  to study the long-time behavior of the solutions to equation \eqref{intro6}, restricted to the domain $G$ (that now is a bounded domain in $\mathbb{R}^d$) and endowed with the boundary condition $u_\e(t,x)=g(x)$, for every $x \in\,\partial G$. In this case
\[u_\e(t,x)=\mathbb{E}g(X^{t,x}_\e(t\wedge \tau^x_\e)),\]
where $\tau^x_\e$ is the first exit time of $X^{t,x}_\e$ from the domain $G$. In particular, the asymptotic description of $\tau^x_\e$ in terms of the quasi-potential has made possible to study  the asymptotic behavior of $u_\e$ on  exponential time scales $t(\e)\sim \exp(\la/\e)$. Freidlin and Koralov's idea is to introduce a family of linear equations obtained from \eqref{intro6} by freezing the second variable in $\sigma\sigma^\star$ and putting it equal to a constant $c$. This allows them to describe the asymptotics of $u_\e(\exp(\la/\e),x)$, for different values of $\lambda \in\,(0,\infty)$, in terms of some function $c(\la)$ obtained from   $V_G(c)$, the minimum of the quasi-potential in $G$ for the linear problem corresponding to $c$, and from $g(x^\star(c))$, where $x^\star(c)$ is the point of $\partial G$ where the quasi-potential attains its minimum, for different values of $c$. 

\medskip

The present paper represents the beginning of a longer term project where we aim  to develop an analogous theory for  infinite dimensional dynamical systems described by PDEs. As in the finite dimensional case studied in \cite{FK}, also here, as a first and fundamental step, we need to be able to study the well-posedness of the following quasi-linear equations
\begin{equation}
\label{quasi-linear}
\left\{\begin{array}{l}
\ds{D_tu_\e(t,x)=\frac \e2 \text{Tr} \left[\sigma\sigma^\star(x,u_\epsilon(t,x))D^2_x	 u_\epsilon(t,x)\right]+\langle Ax+b(x),D u_\epsilon(t,x)\rangle_H,\ \ \ x \in\,H,\ \ \ t>0,}\\[10pt]
\ds{u_\epsilon(0,x)=g(x),\ \ \ x \in\,H,}
	\end{array}\right.
\end{equation}
 However, unlike in finite dimension, where a well-established theory of deterministic quasi-linear PDEs is available, it seems that the current literature does not provide any Hilbert space counterpart to such classical theory, and everything has to be done. 

In our analysis we will proceed in several steps and here  we are considering the case when $\sigma:H\times \mathbb{R}\to\mathcal{L}(H)$ is Lipschitz continuous and there exist a bounded and non-negative symmetric  operator $Q$, a  continuous mapping $f$ defined on $H\times \mathbb{R}$ with values in the space of trace-class operators and a constant $\delta>0$ such that
\[\sigma^\star\sigma(x,r)=Q+\delta\,f(x,r),\ \ \ \ x \in\,H,\ \ r \in\,\mathbb{R}.
\]
This allows to  rewrite equation \eqref{quasi-linear} as
\[\left\{\begin{array}{l}
\ds{D_tu_\e(t,x)=\mathcal{L}_\e u_\e(t,x)+\frac {\e} 2 \text{Tr} \left[\delta\, f(x,u_\e)(t,x)D^2_x	 u_\e(t,x)\right]+\langle b(x),D u_\e(t,x)\rangle_H,}\\[10pt]
\ds{u_\e(0,x)=g(x),\ \ \ x \in\,H,}
	\end{array}\right.\]
where
\[\mathcal{L}_\e\varphi(x)=\frac \e2 \text{Tr}\left[QD_x^2\varphi(x)\right]+\langle Ax, D_x \varphi(x)\rangle_H.\]
In particular, if we denote by $R^\e_t$ the Ornstein-Uhlenbeck semigroup associated with the operator $\mathcal{L}_\e$, we can rewrite equation \eqref{quasi-linear} in mild form as
\begin{equation}
\label{intro8}
u_\e(t,x)=R^\e_tg(x)+\int_0^t R^\e_{t-s}\left(\frac \e 2\text{\mbox{Tr}}\left[\delta F(u_\e(s,\cdot))D^2_x	 u_\e(s,\cdot)\right]+\langle b(\cdot),D u_\e(s,\cdot)\rangle_H\right)(x)\,ds.
	\end{equation}

By assuming suitable conditions on $A$ and $Q$ (see Hypothesis \ref{H2}), the semigroup $R^\e_t$ has a smoothing effect for every fixed $\e>0$, and this allows to prove by a fixed point argument that if $\delta$ is sufficiently small there exists a local solution in a suitable space of smooth H\"older continuous functions. Moreover, we show that any mild solution $u_\e$ defined on a given interval $[0,T]$ is in fact a classical solution. In particular, $u_\e(t,\cdot)$ is twice continuously differentiable in $H$ for every $t>0$ and $QD^2_x u_\e(t,x)$ is a trace-class operator, $u_\e(\cdot,x)$ is differentiable in $(0,+\infty)$ for every $x \in\,D(A)$ and equation \eqref{intro6} holds. Finally, in order to prove that for every $T>0$ and $\e>0$ the solution is defined on the whole interval $[0,T]$ and is unique, we prove suitable a-priori bounds. 

Once we have proven the existence and uniqueness of a classical solution for equation \eqref{intro6}, we can introduce the  stochastic PDE
 \begin{equation}
\label{stoch-pde.intro}
\left\{\begin{array}{l}
\ds{dX_\e^{t,x}(s)=\left[A	X_\e^{t,x}(s)+b(X_\e^{t,x}(s))\right]\,ds+\sqrt{\e}\,\sigma(X_\e^{t,x}(s),u_\e(t-s,X_\e^{t,x}(s)))\,dW_s,}\\[10pt]
\ds{X_\e^{t,x}(0)=x,}
	\end{array}\right.
\end{equation}
where $W_t$ is a cylindrical Wiener process in $H$, defined on some stochastic basis $(\Omega, \mathcal{F}, \{\mathcal{F}_t\}_{t\geq 0}, \mathbb{P})$.
Due to the regularity of the coefficients and of  the function $u_\e$, we can show that  there exists $\bar{\delta}>0$ such that, for every $\delta\leq \bar{\delta}$ and for every  $t>0$ and $x \in\,H$, equation \eqref{stoch-pde.intro} admits a unique mild solution in $L^2(\Omega;C([0,T]; H))$.
Moreover, we show that, as in the finite dimensional case, the quasi-linear equation \eqref{intro6} and the stochastic PDE are related through formula \eqref{maximum} and, in particular,  a maximum principle holds for equation \eqref{quasi-linear}.

%\textcolor{red}{\textbf{[MARIO] }}

It is worth noticing that as a consequence of  the Markov property, the following relation holds 
$$u_\e(t-s,X_\e^{t,x}(s)))=\mathbb{E}(g(X_\e^{t-s,y}(t-s))\big \vert_{y=X_\e^{t,x}(s)}=\mathbb{E}(g(X_\e^{t,x}(t))\vert \mathcal{F}_s),$$
for every $s \in\,[0,t]$ and $x \in\,H$,
so that equation \eqref{intro7} reads as
\begin{equation}
\label{stoch-pde.intro1}
\left\{\begin{array}{l}
\ds{dX_\e^{t,x}(s)=\left[A	X_\e^{t,x}(s)+b(X_\e^{t,x}(s))\right]\,ds+\sqrt{\e}\,\sigma(X_\e^{t,x}(s),\mathbb{E}(g(X_\e^{t,x}(t))\vert \mathcal{F}_s))\,dW_s,}\\[10pt]
\ds{X_\e^{t,x}(0)=x.}
	\end{array}\right.
\end{equation}
Setting $Y^{t, x}(s):=\mathbb{E}(g(X_\e^{t,x}(t))\vert \mathcal{F}_s))$, the  equation above can be further rewritten as a coupled forward backward infinite dimensional stochastic system
\begin{equation}
\label{stoch-fbsde.intro}
\left\{\begin{array}{l}
\ds{dX_\e^{t,x}(s)=\left[A	X_\e^{t,x}(s)+b(X_\e^{t,x}(s))\right]\,ds+\sqrt{\e}\,\sigma(X_\e^{t,x}(s),Y^{t,x}(s)))\,dW_s,\qquad 0\leq s \leq t}\\[10pt]
\ds{-d_s Y^{t,x}(s)= - Z^{t,x}(s) dW_s,\qquad 0\leq s \leq t}\\[10pt]
\ds{Y^{t,x}(t)=g(X_\e^{t,x}(t))}
\\[10pt]
\ds{X_\e^{t,x}(0)=x.}
	\end{array}\right.
\end{equation}
Coupled forward-backward systems of stochastic equations of the general form
\begin{equation}
\label{stoch-fbsde-fd.intro}
\left\{\begin{array}{l}
\ds{dX(s)=b(X(s),Y(s),Z(s))ds+\sigma(X(s),Y(s))\,dW_s,\qquad 0\leq s \leq t}\\[10pt]
\ds{-d_s Y(s)= \psi((X(s),Y(s),Z(s))\,ds - Z(s) dW_s,\qquad 0\leq s \leq t}\\[10pt]
\ds{Y(t)=g(X(t))}
\\[10pt]
\ds{X_\e(0)=x,}
	\end{array}\right.
\end{equation}
have been extensively studied in the finite dimensional case, see \cite{MaYong} where several results are collected. Since \cite{antonelli}, it has been clear that arbitrary forward-backward stochastic systems do not always admit a solution. Different techniques have been developed to prove existence and uniqueness both locally in time and in arbitrarily long time intervals. In particular the classical theory of PDEs, applied to the corresponding nonlinear Kolmogorov equations, offers a wide range of results stating well posedness of system  \eqref{stoch-fbsde-fd.intro} (see, for instance  \cite{MaYong} \cite{gua-del} or \cite{gua-lun}) that include existence and uniqueness of a global solution to the finite dimensional analogue of system \eqref{stoch-fbsde.intro} when $\sigma$ is not degenerate.  In the infinite dimensional case, in which large part of the analytic techniques are not available any more, very few results on existence and uniqueness of a solution to system \eqref{stoch-fbsde.intro} in arbitrary time interval are at hand (for local existence and uniqueness  see \cite{gua}). 
It seems that the techniques more likely to be extended in infinite dimensions are the ones introduced in  \cite{PardouxTang} where  quantitative conditions on dissipativity of $b$ and bounds on the Lipschitz norm of $\sigma$ and $g$ are required.  Such restrictions go in the same direction as the condition on $\delta$ that we have to impose here, see above.

We finally notice that, if we show that system \eqref{stoch-fbsde.intro} is well posed, then we can define a candidate solution to the PDE \eqref{quasi-linear} by setting
 $$\hat{u}_\e(t-s, \xi)=\mathbb{E}(Y^{t,x}(s )| X^{t,x}(s)=\xi)$$ but, unless we have a satisfactory analytic theory for equation \eqref{stoch-fbsde.intro}, the proof that $\hat{u}_\epsilon$ is the unique solution of \eqref{stoch-fbsde.intro} (in which formulation?) is still to be done and does not seem obvious at all.
Once such relation would be understood, it could also be possible to study the large deviations principle, see below,  for more general nonlinear perturbations of \eqref{intro1} defined through systems like \eqref{stoch-fbsde-fd.intro} (see \cite{cgz} for a similar approach in the finite dimensional case  where the connection between $\hat{u}_\epsilon$ and equation \eqref{stoch-fbsde.intro} is a straight-forward consequence of existence and uniqueness of a regular solution to \eqref{stoch-fbsde.intro} and Ito rule).

 \medskip

As we mentioned at the beginning of this introduction, we are interested in applying our results to the study of the asymptotic behavior of \eqref{stoch-pde.intro} and \eqref{quasi-linear}, as $\e\downarrow 0$. This is a multi-step project and here we are addressing the problem of the  validity of a large deviation principle for the trajectories of the solutions of equation \eqref{stoch-pde.intro}. Thus, in the last section of our paper we prove that the family  $\{X_\e^{t,x}\}_{\e \in\,(0,1)}$ satisfies a large deviation principle in the space $C([0,t];H)$, which is governed  by the action functional 
\[I_{t,x}(X)=\frac 12\,\inf\left\{\int_0^t\Vert \varphi(s)\Vert_H^2\,ds\,:\, X(s)=	X^{t,x}_\varphi(s),\ s \in\,[0,t]\right\},
\]
where $X^{t,x}_\varphi$ is the  unique mild solution of  problem
\[X^\prime(s)=A	X(s)+b(X(s))+\sigma(X(s),g(Z^{X(s)}(t-s)))\varphi(s),\ \ \ \ X(0)=x,\]
and for every $y \in\,H$ 
\[Z^y(s)=e^{sA}y+\int_0^s e^{(s-r)A}b(Z^y(r))\,dr.\]
\textcolor{red}{}

\section{Notations and preliminaries}
\label{sec2}

Throughout this paper, $H$ is a separable Hilbert space, endowed with the scalar product $\langle\cdot,\cdot\rangle_H$ and the corresponding norm $\Vert\cdot\Vert_H$. In what follows we shall introduce some notations and preliminary results (we refer to \cite{B}, \cite{DPZ2} and \cite{Kuo} for all  details).

\subsection{Operator spaces} 

We denote by $\mathcal{L}(H)$ the Banach space of  all bounded  linear operators $A:H\to H$, endowed with the sup-norm
\[\Vert A\Vert_{\mathcal{L}(H)}=\sup_{\Vert x\Vert_H\leq 1}\Vert Ax\Vert_H.\] An operator $A \in\,\mathcal{L}(H)$ is symmetric if it coincides with its adjoint $A^\star$, that is if $\langle Ax,y\rangle_H=\langle x,Ay\rangle_H$, for all $x, y \in\,H$. Moreover, it is non-negative if $\langle Ax,x\rangle_H\geq 0$, for all $x \in\,H$.
We shall denote by $\mathcal{L}^+(H)$  the subspace of all non-negative and symmetric operators in $\mathcal{L}(H)$.

An operator $A \in\,\mathcal{L}(H)$ is called an {\em Hilbert-Schmidt} operator if there exists an orthonormal basis $\{e_i\}_{i \in\,\mathbb{N}}$ of $H$ such that
\[\sum_{i=1}^\infty \Vert A e_i\Vert_H^2<\infty.\]
The subspace of Hilbert-Schmidt operators, denoted by $\mathcal{L}_2(H)$, is a Hilbert space, endowed with the scalar product 
\[\langle A,B\rangle_{\mathcal{L}_2(H)}=\sum_{i=1}^\infty \langle A e_i, B e_i\rangle_H.\]

As know, for every  $B \in\,\mathcal{L}^+(H)$ there exists a unique $C \in\,\mathcal{L}(H)$, denoted by $\sqrt{B}$ such that $C^2=B$. Thus, for any $A \in\,\mathcal{L}(H)$ we can define
\[|A|:=\sqrt{A^\star A}.\] We  recall that an operator $A \in\,\mathcal{L}(H)$ is compact if and only if $|A|$ is compact. Moreover, if $A$ is a symmetric compact operator, then there exists an orthonormal basis $\{e_i\}_{i \in\,\mathbb{N}}$ of $H$ and a sequence $\{\a_i\}_{i \in\,\mathbb{N}}$ converging to zero such that $A e_i=\a_i e_i$, for all $i \in\,\mathbb{N}$. With these notations, we say that a compact operator $A \in\,\mathcal{L}(H)$ is {\em nuclear} or {\em trace-class} if there exists an orthonormal basis of $H$ consisting of eigenvectors of $|A|$ corresponding to the eigenvalues $\{\a_i\}_{i \in\,\mathbb{N}}$, such that
\[\sum_{i=1}^\infty \a_i<\infty.\]
In particular, if the operator $A$ is symmetric, it is nuclear if and only if there exists an orthonormal basis of $H$ consisting of eigenvectors of $A$ corresponding to the eigenvalues $\{\a_i\}_{i \in\,\mathbb{N}}$, such that
\[\sum_{i=1}^\infty |\a_i|<\infty.\]
We denote by $\mathcal{L}_1(H)$ the set of nuclear operators.

It is possible to prove that for every $A\in\,\mathcal{L}_1(H)$ the series
\[\text{Tr} A:=\sum_{i=1}^\infty \langle Ae_i,e_i\rangle_H\]
does not depend on the choice of the orthonormal basis $\{e_i\}_{i \in\,\mathbb{N}}$. Moreover, a symmetric operator $A$ belongs to $\mathcal{L}_1(H)$ if and only if the series above converges absolutely for every orthonormal basis $\{e_i\}_{i \in\,\mathbb{N}}$. The space $\mathcal{L}_1(H)$ is a Banach space, endowed with the norm
\[\Vert A\Vert_{\mathcal{L}_1(H)}=\text{Tr}\,|A|,\]
and
\begin{equation}
\label{9}
|\text{Tr} A|\leq \Vert A\Vert_{\mathcal{L}_1(H)}.	
\end{equation}

It is possible to prove that $\mathcal{L}_1(H)\subset \mathcal{L}_2(H)\subset \mathcal{L}(H)$ with
\[\Vert A\Vert_{\mathcal{L}(H)}\leq \Vert A\Vert_{\mathcal{L}_2(H)}\leq \Vert A\Vert_{\mathcal{L}_1(H)},\]
and for $j=1,2$ it holds
\[\Vert AB\Vert_{\mathcal{L}_j(H)}\leq \Vert A\Vert_{\mathcal{L}_j(H)}\Vert B\Vert_{\mathcal{L}(H)},\ \ \ \  \Vert AB\Vert_{\mathcal{L}_j(H)}\leq \Vert B\Vert_{\mathcal{L}_j(H)}\Vert A\Vert_{\mathcal{L}(H)}.\]
Moreover, if $A, B \in\,\mathcal{L}_2(H)$, then $AB \in\,\mathcal{L}_1(H)$, with 
\[\Vert AB\Vert_{\mathcal{L}_1(H)}\leq \Vert A\Vert_{\mathcal{L}_2(H)}\Vert B\Vert_{\mathcal{L}_2(H)}.\]

\subsection{Functional spaces}

If $E$ is an arbitrary Banach space, endowed with the norm $\Vert\cdot\Vert_E$, we denote by $B_b(H;E)$ the space of Borel and bounded functions $\varphi:H\to E$. $B_b(H;E)$ is a Banach space, endowed with the sup-norm
\[\Vert\varphi\Vert_0=\sup_{x \in\,H}\Vert\varphi(x)\Vert_E.\]
Moreover, we denote by 
$C_b(H;E)$ the closed subspace of 
uniformly continuous  and bounded functions.  

For every integer $n\geq 1$, we denote by $C^n_b(H;E)$ the space of all functions $\varphi \in\,C_b(H;E)$ which are $n$-times Fr\'echet differentiable, with uniformly continuous and bounded Fr\'echet derivatives $D^l\varphi:H\to\mathcal{L}^h(H;E)$\footnote{We denote by $\mathcal{L}^l(H;E)$  the space of $l$-linear bounded operators $A:H^l\to E$. When $l=1$, we identify $\mathcal{L}^1(H;\mathbb{R})$ with $H$ and when $l=2$ we identify $\mathcal{L}^2(H;\mathbb{R})$ with $\mathcal{L}(H)$.}, for all $l\leq n$. We have that  $C^k_b(H;E)$ is a Banach space, endowed with the norm
\[\Vert \varphi\Vert_n=\Vert \varphi\Vert_0+\sum_{l=1}^n \Vert D^l\varphi\Vert_0.\]
Next, for every $\vartheta \in\,(0,1)$ we denote by $C^\vartheta_b(H;E)$ the space of all functions $\varphi \in\,C_b(H;E)$ such that
\[\ds{[\varphi]_\vartheta:=\sup_{\substack{x, y \in\,H\\x\neq y}}\frac{\Vert\varphi(x)-\varphi(y)\Vert_E}{\Vert x-y\Vert_H^\vartheta}<\infty.}\]
$C^\vartheta_b(H;E)$ is  a Banach space, endowed with the norm
\[\Vert \varphi\Vert_{\vartheta}=\Vert \varphi\Vert_0+[\varphi]_\vartheta.\]
Finally, for every integer $n \in\,\mathbb{N}$ and $\vartheta \in\,(0,1)$, we denote by $C^{n+\vartheta}_b(H;E)$ the space of all functions $\varphi \in\,C^n_b(H;E)$ such that 
\[[D^n\varphi]_\vartheta:=\sup_{\substack{x, y \in\,H\\x\neq y}}\frac{\Vert D^n\varphi(x)-D^n\varphi(y)\Vert_{\mathcal{L}^n(H;E)}}{\Vert x-y\Vert_H^\vartheta}<\infty.\]
$C^{n+\vartheta}_b(H;E)$ is a Banach space, endowed with the norm
\[\Vert\varphi\Vert_{n+\vartheta}=\Vert \varphi\Vert_0+\sum_{l=1}^n \Vert D^l\varphi\Vert_0+[D^n\varphi]_\vartheta=\Vert \varphi\Vert_n+[D^n\varphi]_\vartheta.\]
Notice that in case $E=\mathbb{R}$, we simply write $B_b(H)$ instead of $B_b(H;E)$, and for every $\a\geq 0$ we write  $C^\a_b(H)$ instead of $C^\a_b(H;\mathbb{R})$.

Now, we want to see how classical interpolatory estimates for functions defined on $\mathbb{R}^n$ are still valid for functions defined on the infinite dimensional Hilbert space $H$. To this purpose, we recall that, as shown in \cite[Theorem 2.3.5]{DPZ2}, for every $0\leq \alpha<\beta<\gamma$ there exists a constant $c=c(\a,\beta,\gamma)>0$ such that for every $\varphi \in\,C^\gamma_b(H)$ 
\begin{equation}
\label{1}
\Vert \varphi\Vert_{\beta}\leq c\,\Vert \varphi\Vert_{\a}^{\frac{\gamma-\beta}{\gamma-\a}}\Vert \varphi\Vert_{\gamma}^{\frac{\beta-\a}{\gamma-\a}}.
\end{equation}
However, in what follows we will need the following additional interpolatory estimates.

\begin{Lemma}
Let us fix  $\vartheta \in\,(0,1)$. Then, for every $\varphi \in\,C^1_b(H)$ we have
\begin{equation}
\label{2}
[\varphi]_\vartheta \leq c_{1,\vartheta}\,\Vert\varphi\Vert_0^{1-\vartheta}\,\Vert D\varphi\Vert_0^\vartheta.
	\end{equation}
	Moreover, for every $\varphi \in\,C^{2+\vartheta}_b(H)$ we have
	\begin{equation}
	\label{3}
	\Vert D^2\varphi\Vert_0\leq c_{2,\vartheta}	\,\Vert D\varphi\Vert_0^{\frac{\vartheta}{1+\vartheta}}\,[D^2\varphi]_\vartheta^{\frac 1{1+\vartheta}},	\end{equation}
	and
\begin{equation}
\label{4}
\Vert D\varphi\Vert_0\leq c_{3,\vartheta}\,\Vert\varphi\Vert_0^{\frac{1+\vartheta}{2+\vartheta}}\,[D^2\varphi]_\vartheta^{	\frac{1}{2+\vartheta}}.
\end{equation}
\end{Lemma}

\begin{proof}
Let us fix $\varphi \in\,C^1_b(H)$ and $x,y \in\,H$. Then, for every $\vartheta \in\,(0,1)$ we have
\[|\varphi(x+y)-\varphi(x)|\leq 2\,\Vert \varphi\Vert_0^{1-\vartheta}\,\left|\int_0^1 \langle D\varphi(x+\lambda y),y\rangle_H\,d\lambda\right|^\vartheta\leq 2\,\Vert \varphi\Vert_0^{1-\vartheta}\,\Vert D\varphi\Vert_0^\vartheta \Vert y\Vert _H^\vartheta,\]
so that \eqref{2} follows. 

Now, if  we fix $\varphi \in\,C^{2+\vartheta}_b(H)$, for every $\mu>0$	and $x, z \in\,H$, with $\Vert z\Vert_H=1$, we have
\begin{equation}
\label{5}
\begin{array}{ll}
\ds{\varphi(x+\mu z)=}  &  \ds{\varphi(x)+\mu\,\langle D\varphi(x),z\rangle_H+\frac {\mu^2} 2\langle D^2\varphi(x)z,z\rangle_H}\\[10pt]
&\ds{+\mu^2\int_0^1(1-r)\langle[D^2\varphi(x+r\mu z)-D^2\varphi(x)]z,z\rangle_H\,dr.	}
\end{array}
\end{equation}
By proceeding as in \cite[proof of Theorem 2.3.5]{DPZ2}, we use \eqref{5} to prove \eqref{3}. Actually, thanks to \eqref{5} we have
\[\begin{array}{ll}
\ds{\frac {\mu^2} 2\left|\langle D^2\varphi(x)z,z\rangle_H\right|}  &  \ds{\leq \left|\varphi(x+\mu z)-\varphi(x)-\mu\,\langle D\varphi(x),z\rangle_H\right|+\mu^{2+\vartheta}[D^2\varphi]_\vartheta\int_0^1(1-r)r^\vartheta\,dr,}\end{array}\]
so that
\[\Vert D^2\varphi\Vert_0\leq \frac 2\mu \,\Vert D\varphi\Vert_0+c_\vartheta \,\mu^\vartheta\,[D^2\varphi]_\vartheta,\ \ \ \ \mu>0.\]
If we take the minimum over $\mu>0$, we get
\eqref{3}.

Finally, by using again \eqref{5}, we have
\[\mu\,\left|\langle D\varphi(x),z\rangle_H\right|\leq \left|\varphi(x+\mu z)-\varphi(x)\right|+\frac {\mu^2} 2\left|\langle D^2\varphi(x)z,z\rangle_H\right|+\mu^{2+\vartheta}[D^2\varphi]_\vartheta\int_0^1(1-r)r^\vartheta\,dr,\]
so that, in view of \eqref{3}, we get
\[\begin{array}{ll}
\ds{\Vert D\varphi\Vert_0}& \ds{\leq \frac 2\mu\,\Vert\varphi\Vert_0+\frac{c_{2,\vartheta}\mu}2\,\Vert D\varphi\Vert_0^{\frac{\vartheta}{1+\vartheta}}\,[D^2\varphi]_\vartheta^{\frac 1{1+\vartheta}}+c_\vartheta \mu^{1+\vartheta}[D^2\varphi]_\vartheta}\\[10pt]
&\ds{\leq \frac 2\mu\,\Vert\varphi\Vert_0+\frac 12 \Vert D\varphi\Vert_0+ c_\vartheta \mu^{1+\vartheta}[D^2\varphi]_\vartheta.}
	\end{array}\]
This implies that
\[\Vert D\varphi\Vert_0\leq \frac{4}{\mu} \,\Vert\varphi\Vert_0+
c_\vartheta \mu^{1+\vartheta}[D^2\varphi]_\vartheta,\ \ \ \ \ \mu>0,\]
and if we minimize once again with respect to $\mu>0$ we obtain \eqref{4}.
\end{proof}

\begin{Remark}
{\em As a consequence of \eqref{2}, \eqref{3} and \eqref{4}, we have that for every $\vartheta \in\,(0,1)$ there exists some $c_\vartheta>0$ such that for every $\varphi \in\,C^{2+\vartheta}_b(H)$
\begin{equation}
\label{6}
[\varphi]_\vartheta \,\Vert D^2\varphi\Vert_0\leq c_\vartheta\,\Vert \varphi\Vert_0\,[D^2\varphi]_\vartheta.	
\end{equation}
Actually, from \eqref{3} and \eqref{4}, we have
\begin{equation}
\label{7}
\Vert D^2\varphi\Vert_0\leq c_{2,\vartheta}	\left(c_{3,\vartheta}\,\Vert\varphi\Vert_0^{\frac{1+\vartheta}{2+\vartheta}}\,[D^2\varphi]_\vartheta^{	\frac{1}{2+\vartheta}}\right)^{\frac\vartheta{1+\vartheta}}\,[D^2\varphi]_\vartheta^{\frac 1{1+\vartheta}}\leq c_{2,\vartheta}\,c_{3,\vartheta}^{\frac 1{1+\vartheta}}\,\Vert\varphi\Vert_0^{\frac{\vartheta}{2+\vartheta}}\,[D^2\varphi]_\vartheta^{\frac 2{2+\vartheta}}.	
\end{equation}
Moreover, thanks to \eqref{2} and \eqref{4} we have
\[[\varphi]_\vartheta \leq c_{1,\vartheta}\,\Vert\varphi\Vert_0^{1-\vartheta}\,\left(c_{3,\vartheta}\,\Vert\varphi\Vert_0^{\frac{1+\vartheta}{2+\vartheta}}\,[D^2\varphi]_\vartheta^{	\frac{1}{2+\vartheta}}\right)^\vartheta=c_{1,\vartheta}\,c_{3,\vartheta}^\vartheta\,\Vert\varphi\Vert_0^{\frac{2}{2+\vartheta}}\,[D^2\varphi]_\vartheta^{	\frac{\vartheta}{2+\vartheta}}.\]
Therefore, if we combine together this last inequality with \eqref{7}, we obtain \eqref{6}.

}	
\end{Remark}

\subsection{The Ornstein-Uhlenbeck semigroup}
\label{ssec2.3}
By following \cite[Chapter 6]{DPZ2}, we recall here some results about the Ornstein-Uhlenbeck semigroup and the associated Kolmogorov equation. 

Let  $A:D(A)\subset H\to H$ be the generator of a $C_0$-semigroup $e^{tA}$ and  let $Q$ be an operator in $\mathcal{L}^+(H)$. For every $t\geq 0$ we define
\[Q_t :=\int_0^t e^{sA}Qe^{sA^\star}\,ds,\]
and we assume that $Q_t \in\,\mathcal{L}_1(H)$, for every $t\geq 0$. Thus, we can introduce the centered Gaussian measure $\mathcal{N}_{Q_t}$ defined on $H$ with covariance $Q_t$, and  we can define 
\begin{equation}
\label{10}
R_t\varphi(x):=\int_H\varphi(e^{tA}x+y)\,\mathcal{N}_{Q_t}(dy),\ \ \ \ \ x \in\,H,\ \ \ t\geq 0,	
\end{equation}
for every $\varphi$ in $B_b(H)$.
 $R_t$ is   the {\em Ornstein-Uhlenbeck} semigroup associated with $A$ and $Q$.
In what follows, we  assume that 
 \begin{equation}
\label{8}
e^{tA}(H)\subset Q_t^{1/2}(H),\ \ \ \ t>0.	\end{equation}
and we define
\[\Lambda_t:=Q_t^{-1/2}e^{tA},\ \ \ \ t>0.\]
As shown e.g. in \cite[Theorem 6.2.2]{DPZ2}, as a consequence of assumption \eqref{8}  we have that $R_t\varphi \in\,C^\infty_b(H)$, for every $\varphi \in\,B_b(H)$ and $t>0$, and for every $n \in\,\mathbb{N}\cup \{0\}$ there exists some $c_n>0$ such that
\begin{equation}
\label{61}	
\Vert D^nR_t\varphi\Vert_0\leq c_n\,\Vert \Lambda_t\Vert_{\mathcal{L}(H)}^n\,\Vert \varphi\Vert_0.	
\end{equation}
Moreover, if we fix $\a \in\,(0,1)$ and assume $\varphi \in\,C^\a_b(H)$ we have
\begin{equation}
\label{78}	
[D^nR_t\varphi]_\a\leq c_n\,\Vert \Lambda_t\Vert_{\mathcal{L}(H)}^n\,\Vert e^{tA}\Vert_{\mathcal{L}(H)}^\a\,[\varphi]_\a,
\end{equation}
so that we conclude that for all $\a \in\,[0,1)$ and $\varphi \in\,C^\a_b(H)$
\begin{equation}
\label{60}	
\Vert D^nR_t\varphi\Vert_\a \leq c_n\,\Vert \Lambda_t\Vert_{\mathcal{L}(H)}^n\,\Vert \varphi\Vert_{t,\a},\ \ \ \ \ t>0,	
\end{equation}
where, for every $t\geq 0$ and $\a \in\,(0,1)$,
\begin{equation}
\label{pht_ta}
\Vert \varphi\Vert_{t,\a}:=\left(\Vert \varphi\Vert_0+\Vert e^{tA}\Vert_{\mathcal{L}(H)}^\a\,[\varphi]_\a\right),	
\end{equation}

By the interpolation inequality \eqref{1}, for every $n \in\,\mathbb{N}\cup \{0\}$ and $0\leq\a\leq \beta<1$ and for every $\varphi \in\,C^\beta_b(H)$ and $t>0$, we have
\[\begin{array}{ll}
\ds{\Vert D^n R_t\varphi\Vert_{1+\a}}  &  \ds{\leq \Vert D^n R_t\varphi\Vert_\beta^{\beta-\a}\,\Vert D^n R_t\varphi\Vert_{1+\beta}^{1-(\beta-\a)}}\\[10pt]
&\ds{=\Vert D^n R_t\varphi\Vert_\beta^{\beta-\a}\left(\Vert D^n R_t\varphi\Vert_0+\Vert D^{n+1} R_t\varphi\Vert_\beta\right)^{1-(\beta-\a)}.}	
\end{array}
\]
Hence, 
thanks to \eqref{60}, we get
\[\begin{array}{ll}
\ds{\Vert D^n R_t\varphi\Vert_{1+\a}}  &  \ds{\leq c_{n}^{\,\beta-\a}\Vert \Lambda_t\Vert_{\mathcal{L}(H)}^{n(\beta-\a)}\,\Vert \varphi\Vert_{t,\beta}^{\beta-\a}\left(c_n\,\Vert \Lambda_t\Vert_{\mathcal{L}(H)}^{n}\,\Vert \varphi\Vert_0+c_{n+1}\,\Vert \Lambda_t\Vert_{\mathcal{L}(H)}^{n+1}\Vert \varphi\Vert_{t,\beta}\right)^{1-(\beta-\a)}}\\[10pt]
&\ds{\leq c_{\a,\beta,n}\,\Vert \varphi\Vert_{t,\beta}\left(\Vert \Lambda_t\Vert_{\mathcal{L}(H)}^n+\Vert \Lambda_t\Vert_{\mathcal{L}(H)}^{n+1-(\beta-\a)}\right).}	
\end{array}\]
In particular, this allows to conclude that for every $n \in\,\mathbb{N}$ and $0\leq \a\leq \beta\leq 1$ and for every $\varphi \in\,C_b^\beta(H)$ 
\begin{equation}
\label{65}
\Vert D^{n}R_t\varphi\Vert_{\a}\leq c_{\a,\beta,n}\,\left(\Vert \Lambda_t\Vert_{\mathcal{L}(H)}^{n-1}+\Vert \Lambda_t\Vert_{\mathcal{L}(H)}^{n-(\beta-\a)}\right)\,\Vert \varphi\Vert_{t,\beta}.	
\end{equation}

\medskip

Next, we recall that in \cite[Proposition 6.2.9]{DPZ2} it is shown that for every $\varphi \in\,C^1_b(H)$ and $x \in\,H$
\[ \text{Tr}\,[QD^2R_t\varphi(x)]=\int_H\langle Q_t^{-1/2} y,\Lambda_tQe^{tA^\star}D\varphi(e^{tA}x+y)\rangle_H\,\mathcal{N}_{Q_t}(dy),\]
so that, if we assume that 
\begin{equation}
\label{**}	
\Lambda_tQe^{tA^\star} \in\,\mathcal{L}_2(H),
\end{equation}
 we have
\[\sup_{x \in\,H}\Vert QD^2R_t\varphi(x)\Vert_{\mathcal{L}_1(H)}\leq \Vert \Lambda_tQe^{tA^\star}\Vert_{\mathcal{L}_2(H)}\,\Vert D\varphi\Vert_0,\ \ \ \ \ t>0.	
\]
Therefore, thanks to the semigroup law and \eqref{65} we conclude that for every  $\varphi \in\,C^\beta_b(H)$ 
\begin{equation}
\label{16}
\sup_{x \in\,H}\Vert QD^2R_t\varphi(x)\Vert_{\mathcal{L}_1(H)}\leq 
c\,\Vert \Lambda_{t/2}Qe^{tA^\star/2}\Vert_{\mathcal{L}_2(H)}\Vert \Lambda_{t/2}\Vert_{\mathcal{L}(H)}^{1-\beta}\Vert \varphi\Vert_{\beta},\ \ \ \ \ t>0.	\end{equation}

Moreover,  we recall that in \cite[Proposition 6.2.5]{DPZ2} it is shown that if the operator $\Lambda_tA$ has a continuous extension $\overline{\Lambda_tA}$ to $H$, for every $t>0$, then for every $\varphi \in\,B_b(H)$ and $x \in\,H$
\begin{equation}
\label{17}
DR_t\varphi(x) \in\,	D(A^\star),\ \ \ \ \Vert A^\star DR_t\varphi\Vert_0\leq \Vert \overline{\Lambda_tA}\Vert_{\mathcal{L}(H)}\,\Vert \varphi\Vert_0,\ \ \ \ t>0.
\end{equation}

Now, we introduce the parabolic equation in $H$
\begin{equation}
\label{OU}
D_tu(t,x)=\frac 12 \text{Tr}\left[QD_x^2u(t,x)\right]+\langle x,A^\star D_x u(t,x)\rangle_H,\ \ \ \ \ u(0,x)=\varphi(x).	
\end{equation}
\begin{Definition}
A function $u:[0,+\infty)\times H\to \mathbb{R}$ is a {\em classical solution} of problem \eqref{OU} if
\begin{enumerate}
\item[1.] $u$ is continuous in 	$[0,+\infty)\times H$ and $u(0,\cdot)=\varphi$.
\item[2.] $u(t,\cdot) \in\,C^2_b(H)$, for all $t>0$,   and $QD^2_xu(t,x) \in\,\mathcal{L}_1(H)$, for all $t>0$ and $x \in\,H$.
\item[3.] $D_x u(t,x) \in\,D(A^\star)$, for all $t>0$ and $x \in\,H$.
\item[4.] $u(\cdot,x)$ is differentiable in $(0,+\infty)$ for every $x \in\,H$ and $u$ satisfies equation \eqref{OU}.
\end{enumerate}

	\end{Definition}
	
	In \cite[Theorem 6.2.4]{DPZ2} it is shown that if we assume conditions \eqref{8} and \eqref{**} and we assume that the operator $\Lambda_tA$ has a continuous extension to $H$, then for every $\varphi \in\,B_b(H)$ the function 
	\[u(t,x)=R_t\varphi(x)\]
	is the unique classical solution of equation \eqref{OU}.

\section{Assumptions and main results}
\label{sec3}

\subsection{Assumptions} In what follows, we shall make the following hypotheses.
\begin{Hypothesis}
\label{H1}
\begin{enumerate}
\item The mapping $\sigma:H\times \mathbb{R}\to\mathcal{L}(H)$ is Lipschitz continuous and there exist an operator $Q \in\,\mathcal{L}^+(H)$, a  continuous mapping $f:H\times \mathbb{R}\to\mathcal{L}_1(H)$ and a constant $\delta>0$ such that
\begin{equation}
\label{34}	
\sigma^\star\sigma(x,r)=Q+\delta\,f(x,r),\ \ \ \ x \in\,H,\ \ r \in\,\mathbb{R}.
\end{equation}

\item For every fixed $x \in\,H$, the function $f(x,\cdot):\mathbb{R}\to \mathcal{L}_1(H)$ is differentiable. Both $f$ and $\partial_r f$  are Lipschitz continuous in both variables, uniformly with respect to the other. Moreover
	\begin{equation}
	\label{29-bis}
	\sup_{x \in\,H}\Vert f(x,r)\Vert_{\mathcal{L}_1(H)}\leq c\left(1+|r|\right),\ \ \ \ \ \ r \in\,\mathbb{R}.	
	\end{equation}
\end{enumerate}
\end{Hypothesis}

\begin{Remark}
{\em  Let $H=L^2(\mathcal{O})$, for some smooth and bounded domain $\mathcal{O}\subset \mathbb{R}^d$, with $d\geq 1$. Let $\{e_i\}_{i \in\,\mathbb{N}}$ be an orthonormal basis of $H$ and let $\{\lambda_i\}_{i \in\,\mathbb{N}}$ be a sequence of non-negative real numbers. We assume  that $e_i \in\,L^\infty(\mathcal{O})$, for every $i \in\,\mathbb{N}$, and
\begin{equation}
\label{32}
\sum_{i=1}^\infty \la_i\,\Vert e_i\Vert_{L^\infty(\mathcal{O})}<\infty.	\end{equation}
For every $x, y \in\,H$ and $r \in\,\mathbb{R}$, we define 
	\[[f(x,r)y](\xi)=\sum_{i=1}^\infty \mathfrak{f}_i(x(\xi),r)\lambda_i \langle y,e_i\rangle_H e_i(\xi),\ \ \ \ \ \xi \in\,\mathcal{O},\] for some  continuous functions $\mathfrak{f}_i:\mathbb{R}\times \mathbb{R}\to \mathbb{R}$ such that $\mathfrak{f}_i(s,\cdot):\mathbb{R}\to\mathbb{R}$ is  differentiable, for every $s \in\,\mathbb{R}$ and $i \in\,\mathbb{N}$. We assume that  both $\mathfrak{f}_i$ and $\partial_r \mathfrak{f}_i$ are Lipschitz continuous in both variables, uniformly with respect to the other variable, and uniformly with respect to $i \in\,\mathbb{N}$. Moreover, we assume that 
	\begin{equation}
	\label{30}
	\sup_{i \in\,\mathbb{N}}\,\sup_{s \in\,\mathbb{R}}|\mathfrak{f}_i(s,r)|\leq c\left(1+|r|\right),\ \ \ \ \ \ r \in\,\mathbb{R}.	
	\end{equation}
With this choice of $H$ and $f$, we have that  condition 2 in Hypothesis \ref{H1} holds.
	
Actually, since $\mathfrak{f}_i(\cdot,r):\mathbb{R}\to\mathbb{R}$ is Lipschitz continuous, uniformly with respect to $r \in\,\mathbb{R}$ and $i \in\,\mathbb{N}$, for every $x, y\in\,H$ and $r \in\,\mathbb{R}$ we have
\[\begin{array}{l}
\ds{\sum_{i=1}^\infty |\langle \left[f(x,r)-f(y,r)\right]e_i,e_i\rangle_H|=\sum_{i=1}^\infty \lambda_i \,|\langle\left[\mathfrak{f}_i(x(\cdot),r)-\mathfrak{f}_i(y(\cdot),r) \right]e_i,e_i\rangle_H|}\\[18pt]
\ds{\leq \sum_{i=1}^\infty \Vert \mathfrak{f}_i(x(\cdot),r)-\mathfrak{f}_i(y(\cdot),r)\Vert_H \Vert e_i\Vert_{L^\infty(\mathcal{O})}\lambda_i\leq c\,\Vert x-y\Vert_H\,\,\sum_{i=1}^\infty \lambda_i \Vert e_i\Vert_{L^\infty(\mathcal{O})}.}	
\end{array}\]
In particular, thanks to  \eqref{32}, we can conclude that $f(\cdot,r):H\to \mathcal{L}_1(H)$ is Lipschitz continuous, uniformly with respect to $r \in\,\mathbb{R}$. In view of our assumptions, the same is true for $\partial_r f$. 

The Lipschitz continuity of $f(x,\cdot),\ \partial_r f(x,\cdot):\mathbb{R}\to \mathcal{L}_1(H)$, uniform with respect to $x \in\,H$, is proved in a similar way. However, in this case  \eqref{32} is not required and we only need the weaker condition
\[\sum_{i=1}^\infty \la_i<\infty.\]
Finally, \eqref{29-bis} is an immediate consequence of \eqref{30}.

}	
\end{Remark}

Now, we see some consequences of Hypothesis \ref{H1}.

\begin{Lemma}
For any function $\varphi:H\to\mathbb{R}$ we define
\begin{equation}
\label{33}	
F(\varphi)(x)=f(x,\varphi(x)),\ \ \ \ x\in\,H.
\end{equation}
Then, under Hypothesis \ref{H1} we have that $F$ maps $C^\vartheta_b(H)$ into $C^\vartheta_b(H;\mathcal{L}_1(H))$ and for every $\varphi \in\,C^\vartheta_b(H)$
\begin{equation}
\label{25}
\Vert F(\varphi)\Vert_\vartheta\leq c\left(1+\Vert \varphi\Vert_\vartheta\right).	
\end{equation}
Moreover for every $\varphi_1,\ \varphi_2 \in\,C^\vartheta_b(H)$ it holds
\begin{equation}
\label{28}
\Vert F(\varphi_1)-F(\varphi_2)\Vert_{\vartheta}\leq c\left(1+\Vert \varphi_1\Vert_\vartheta+\Vert \varphi_2\Vert_\vartheta\right)\Vert \varphi_1-\varphi_2\Vert_{\vartheta}.	\end{equation}
\end{Lemma}

\begin{proof}
Due to \eqref{29-bis}, if $\varphi \in\,C^\vartheta_b(H)$  we have
	\begin{equation}
	\label{26}
	\Vert F(\varphi)\Vert_{0}\leq \sup_{x \in\,H}\Vert f(x,\varphi(x))\Vert_{\mathcal{L}_1(H)}\leq c\left(1+\Vert \varphi\Vert_0\right).	
	\end{equation}
	Moreover,
	for every $x, y \in\,H$
	\[\begin{array}{l}
\ds{	\Vert f(x,\varphi(x))-f(y,\varphi(y))\Vert_{\mathcal{L}_1(H)}}\\[10pt]
\ds{\leq \Vert f(x,\varphi(x))-f(y,\varphi(x))\Vert_{\mathcal{L}_1(H)}^\vartheta \,\Vert f(x,\varphi(x))-f(y,\varphi(x))\Vert_{\mathcal{L}_1(H)}^{1-\vartheta}}\\[10pt]
\ds{+\Vert f(y,\varphi(x))-f(y,\varphi(y))\Vert_{\mathcal{L}_1(H)}\leq c\,\Vert x-y\Vert_H^\vartheta\,\left(1+\Vert \varphi\Vert_0^{1-\vartheta}\right)+c\,|\varphi(x)-\varphi(y)|}\\[10pt]
\ds{\leq c\,\Vert x-y\Vert_H^\vartheta\,\left(1+\Vert \varphi\Vert_0^{1-\vartheta}+[\varphi]_\vartheta \right)\leq c\,\Vert x-y\Vert_H^\vartheta\,\left(1+\Vert \varphi\Vert_0+[\varphi]_\vartheta \right),}
	\end{array}\]
	so that 
	\begin{equation}
	\label{27}
	[F(\varphi)]_\vartheta	\leq c\,\left(1+\Vert \varphi\Vert_0+[\varphi]_\vartheta \right).
	\end{equation}
This, together with \eqref{26} allows to   conclude that $F(\varphi) \in\,C^\vartheta_b(H)$ and \eqref{25}  holds.
	
Concerning \eqref{28}, for every $\varphi_1, \varphi_2 \in\,C_b^\vartheta(H)$ we have	
	\begin{equation}
	\label{29}
\begin{array}{ll}
\ds{\Vert F(\varphi_1)-F(\varphi_2)\Vert_{0}} & \ds{\leq c\sup_{x \in\,H}\Vert f(x,\varphi_1(x))-f(x,\varphi_2(x))\Vert_{\mathcal{L}_1(H)}\,\leq c\,\Vert \varphi_1-\varphi_2\Vert_0.}	
\end{array}		
	\end{equation}
Moreover, for every $x, y \in\,H$ we have

\[\begin{array}{l}
\ds{(f(x,\varphi_1(x))-f(x,\varphi_2(x)))-(f(y,\varphi_1(y))-f(y,\varphi_2(y)))}	\\[10pt]
\ds{=\int_0^1\left[\gamma(s\varphi_1+(1-s)\varphi_2)(x)(\varphi_1-\varphi_2)(x)-\gamma(s\varphi_1+(1-s)\varphi_2)(y)(\varphi_1-\varphi_2)(y)\right]\,ds,}
\end{array}\]
where we have defined
\[\gamma(\varphi)(x)=\partial_rf(x,\varphi(x)),\ \ \ \ \ x \in\,H.\]
This implies that
\[\begin{array}{ll}
\ds{[ f(\cdot,\varphi_1)-f(\cdot,\varphi_2)]_\vartheta\leq \int_0^1\Vert\gamma(s\varphi_1+(1-s)\varphi_2)\Vert_{\vartheta}\,ds\,\Vert\varphi_1-\varphi_2\Vert_\vartheta.}\end{array}\]
Since, we are assuming that $\partial_r f$, like  $f$, is Lipschitz continuous with respect to each variable, uniformly with respect to the other, and  is clearly uniformly bounded, by using the same arguments we have used to prove \eqref{26} and \eqref{27}, we have
\[\Vert \gamma(s\varphi_1+(1-s)\varphi_2)\Vert_{\vartheta}\leq c\left(1+s\,\Vert \varphi_1\Vert_\vartheta+(1-s)\,\Vert \varphi_2\Vert_\vartheta\right),\]
and hence
\[[F(\varphi_1)-F(\varphi_2)]_\vartheta\leq 
c\left(1+\Vert \varphi_1\Vert_\vartheta+\Vert \varphi_2\Vert_\vartheta\right)\Vert \varphi_1-\varphi_2\Vert_\vartheta.\]
This, together with \eqref{29}, implies \eqref{28}.
	
\end{proof}

\begin{Hypothesis}
	\label{H2}
	\begin{enumerate}
\item The operator $A:D(A)\subset H\to H$ generates  a $C_0$-semigroup $e^{tA}$ and there exist $M, \omega>0$ such that
\begin{equation}
\label{67}
\Vert e^{tA}\Vert_{\mathcal{L}(H)}\leq M e^{-\omega t}.	
\end{equation}

\item 	If  $Q$ is the operator introduced in Hypothesis \ref{H1} and if we define
\[Q_t :=\int_0^t e^{sA}Qe^{sA^\star}\,ds,\ \ \ \ \ t\geq 0,\]
we have that $Q_t \in\,\mathcal{L}^+_1(H)$, for every $t\geq 0$.	
\item For every $t>0$, we have
\begin{equation}
\label{21}
e^{tA}(H)\subset Q_t^{1/2}(H).	
\end{equation}
\item If we define
\[\Lambda_t:=Q_t^{-1/2}e^{tA},\ \ \ \ t>0,\]
there exists some $\la>0$ such that
\begin{equation}
\label{20}
\Vert \Lambda_t\Vert_{\mathcal{L}(H)}\leq c\,(t\wedge 1)^{-1/2}e^{-\la t},\ \ \ \ \ t>0.	
\end{equation}
\item For every $t>0$ we have that $\Lambda_{t}Qe^{tA^\star} \in\,\mathcal{L}_2(H)$ and for every $\vartheta \in\,(0,1)$ there exist $\beta_\vartheta <1$ and $\a_\vartheta>0$ such that
\begin{equation}
\label{100}	
 \kappa_\vartheta(t):=\Vert \Lambda_{t}Qe^{tA^\star}\Vert_{\mathcal{L}_2(H)}\Vert \Lambda_{t}\Vert_{\mathcal{L}(H)}^{1-\vartheta} \leq c\,(t\wedge 1)^{-\beta_\vartheta} e^{-\a_\vartheta t},\ \ \ \ t>0.
\end{equation}
\end{enumerate}
\end{Hypothesis}

\begin{Hypothesis}
\label{H3}	
For every $(x,r) \in\,H\times \mathbb{R}$ and $t>0$ we have  
\[e^{tA}\sigma(x,r) \in\,\mathcal{L}_2(H).\]
Moreover, 
\[\Vert e^{tA}\sigma(x,r)\Vert_{\mathcal{L}_2(H)}\leq c\,(t\wedge 1)^{-\frac 14}\,\left(1+\Vert x\Vert_H+|r|\right),\ \ \ \ t>0,\]
and for every $(x,r), (y,s) \in\,H\times \mathbb{R}$
\[\Vert e^{tA}\sigma(x,r)-e^{tA}\sigma(y,r)\Vert_{\mathcal{L}_2(H)}\leq c\,(t\wedge 1)^{-\frac 14}\,\left(\Vert x-y\Vert_H+|r-s|\right),\ \ \ \ t>0.\]
\end{Hypothesis}

\begin{Remark}
{\em Let $\{e_i\}_{i \in\,\mathbb{N}}$ be an orthonormal basis in $H$ and assume that $A e_i=-\a_i e_i$ and $Q e_i=\gamma_i\,e_i$, for every $i \in\,\mathbb{N}$, with $\a_i,\ \gamma_i >0$, and $\a_i\uparrow +\infty$, as $i\to\infty$. By proceeding as in \cite[Example 6.2.11]{DPZ2}, we have that 
\[Q_t e_i=\frac{\gamma_i}{2\a_i}\left(1-e^{-2\a_i t}\right)e_i,\ \ \ \ i \in\,\mathbb{N},\]
so that $Q_t \in\,\mathcal{L}_1(H)$ if and only if 
\begin{equation}
\label{22}	\sum_{i=1}^\infty \frac{\gamma_i}{\a_i}<\infty.
\end{equation}
Moreover, 
\[\Lambda_t e_i=\left(\frac{2\,\a_i t\,e^{-\a_i t}}{\gamma_i\,(1-e^{-2\a_i t})}\right)^{1/2}t^{-1/2}e^{-\frac{\a_i}2 t}e_i,\ \ \ \ i \in\,\mathbb{N}.\]
In particular, if $\gamma_i\geq \gamma_0>0$, we have
\[
\Vert \Lambda_t\Vert_{\mathcal{L}(H)}\leq c\,t^{-1/2}e^{-\frac{\a_1}2 t},\ \ \ \ \ t>0,	
\]
so that \eqref{20} holds. Furthermore, 
\begin{equation}
\label{24}	
\Vert \Lambda_t Q e^{tA^\star}\Vert^2_{\mathcal{L}_2(H)}=2\sum_{i=1}^\infty \frac{\a_i\,\gamma_i\, e^{-2\a_i t}}{e^{2\a_i t}-1}\leq c\,t^{-1}e^{-2\a_1 t},
\end{equation}

When $A$ is the realization of the Laplace operator in an interval, endowed with Dirichlet boundary conditions, we have that $\a_i\sim i^2$ and \eqref{22} is satisfied, for every choice of $Q \in\,\mathcal{L}(H)$. 
If we assume that $Q=I$,  we have that \eqref{20} holds. Moreover,   thanks to \eqref{24}  we have
\[\Vert \Lambda_t Q e^{tA^\star}\Vert_{\mathcal{L}_2(H)}\Vert \Lambda_{t}\Vert_{\mathcal{L}(H)}^{1-\vartheta}\leq c\, t^{-1/2}\,e^{-\a_1 t}\,t^{-(1-\vartheta)/2}e^{-\frac{\a_1(1-\vartheta)t}2}=c\,t^{-(1-\vartheta/2)}e^{-\frac{\a_1(3-\vartheta)t}2},\]
and Condition (5) in Hypothesis \ref{H2} holds for every $\vartheta \in\,(0,1)$.
Notice also that in this case Hypothesis \ref{H3} is satisfied.}

\end{Remark}

 \begin{Hypothesis}
\label{H4}
The mapping $b:H\to H$ is Lipschitz continuous and bounded.		
\end{Hypothesis}

\bigskip

\subsection{Main results}
As we have done in Section \ref{sec2} for the linear Kolmogorov equation \eqref{OU}, we introduce here the notion of {\em classical solution} for the quasi-linear problem 
\begin{equation}
\label{quasi-linear-bis}
\left\{\begin{array}{l}
\ds{D_tu_\e(t,x)=\frac \e2 \text{Tr} \left[\sigma^\star\sigma(x,u_\epsilon(t,x))D^2_x	 u_\epsilon(t,x)\right]+\langle Ax+b(x),D u_\epsilon(t,x)\rangle_H,\ \ \ x \in\,H,\ \ \ t>0,}\\[10pt]
\ds{u_\epsilon(0,x)=g(x),\ \ \ x \in\,H.}
	\end{array}\right.
\end{equation}

\begin{Definition}
\label{def-classical}
A function $u_\e:[0,+\infty)\times H\to \mathbb{R}$ is a {\em classical solution} of problem \eqref{quasi-linear-bis} if the following conditions are satisfied.
\begin{enumerate}
\item[1.] It is continuous in 	$[0,+\infty)\times H$ and $u_\e(0,\cdot)=g$.
\smallskip
\item[2.] $u_\e(t,\cdot) \in\,C^{2}_b(H)$, for all $t>0$,   and $QD^2_xu_\e(t,x) \in\,\mathcal{L}_1(H)$, for all $(t,x) \in\,(0,+\infty)\times H$.
\smallskip

\item[3.] $u_\e(\cdot,x)$ is differentiable in $(0,+\infty)$, for every $x \in\,D(A)$. 
\smallskip

\item[4.] It satisfies equation \eqref{quasi-linear-bis}, for every   $(t,x) \in\,(0,+\infty)\times D(A)$.
\end{enumerate}
	
\end{Definition}

In what follows, for every $\e \in\,(0,1)$, $0<\vartheta< \eta<1$, $\varrho \in\,(0,1/2)$   and $T>0$,  we denote by 
$C_{\e,\varrho, \eta}((0,T];C^{2+\vartheta}_b(H))$ the space of all functions $ u \in\,C([0,T];C^\eta_b(H))\cap C((0,T];C^{2+\vartheta}_b(H))$ such that
\[\Vert u\Vert_{\e, \varrho, \eta, \vartheta, T}:=\sup_{t \in\,(0,T]}\left(\Vert u(t,\cdot)\Vert_\eta+\e^{\varrho}(t\wedge 1)^{\varrho}\Vert D_xu(t,\cdot)\Vert_{\vartheta}+\e^{\varrho+\frac 12}(t\wedge 1)^{\varrho+\frac 12}\Vert D^2_xu(t,\cdot)\Vert_\vartheta\right)<\infty.\]

\begin{Theorem}
\label{teo1}
Assume Hypotheses \ref{H1} to \ref{H4}, and fix  an arbitrary $g \in\,C^\eta_b(H)$, for some $\eta >1/2$. Then	there exists $\bar{\delta}>0$ such that for every $\delta\leq \bar{\delta}$ and $\e \in\,(0,1)$ there exists a unique classical solution $u_\e$ for equation \eqref{quasi-linear-bis}. Moreover, if we fix $\vartheta \in\,(0,(\eta-1/2)\wedge 1)$ and we define
\[\varrho=\frac{1-(\eta-\vartheta)}2,
\]
we have  that 
$u_\e \in\,C_{\e,\varrho, \eta}((0,T];C^{2+\vartheta}_b(H))$, 
for every $T>0$ and $\e \in\,(0,1)$, and 
\begin{equation}
	\Vert u_\e\Vert_{\e, \varrho, \eta, \vartheta, T}\leq c_\e\,\Vert g\Vert_\eta,\ \ \ \ \ \e \in\,(0,1),
\end{equation}
for some constant $c_\e>0$ independent of $T>0$.

\end{Theorem}

\bigskip
Once proved the existence and uniqueness of a classical solution $u_\e$ for problem \eqref{quasi-linear-bis}, for every $\e>0$, we fix arbitrary $t>0$ and $x \in\,H$ and we introduce the following stochastic PDE
\begin{equation}
\label{stoch-pde}
\left\{\begin{array}{rl}
\ds{dX(s)=} & \ds{\left[A	X(s)+b(X(s))+ \sigma(X(s),u_\e(t-s,X(s)))\,\varphi(s)\right]\,ds}\\[14pt]
& \ds{+\sqrt{\e}\,\sigma(X(s),u_\e(t-s,X(s)))\,dW_s,}\\[10pt]
\ds{X(0)=} & \ds{x.}
	\end{array}\right.
\end{equation}
Here $W_t$, $t\geq 0$, is a cylindrical Wiener process on  $H$, defined  on  the filtered probability space $(\Omega, \mathcal{F}, \{\mathcal{F}_t\}_{t\geq 0}, \mathbb{P})$,  such that for every $h, k \in\,H$ and $t, s\geq 0$
\[\mathbb{E}\,\langle W_t,h\rangle_H\langle W_s,k\rangle_H=(t\wedge s)\,\langle h,k\rangle_H,\]
 and $\varphi$ is a predictable process in $L^2(\Omega;L^2(0,t;H))$.
\begin{Definition}
An adapted process $X^{t, x}_{\varphi,\e} \in\,L^2(\Omega;C([0,t];H))$ is a {\em mild solution} for equation \eqref{stoch-pde} if for every $s \in\,[0,t]$
\begin{equation}
\label{35}
\begin{array}{ll}
\ds{	X_{\varphi, \e}^{t,x}(s)=}  &  \ds{e^{sA}x+\int_0^s e^{(s-r)A}b(X_{\varphi, \e}^{t,x}(r))	\,dr+\int_0^se^{(s-r)A}\sigma(X_{\varphi, \e}^{t,x}(r),u_\e(t-r,X_{\varphi, \e}^{t,x}(r)))\,\varphi(r)\,dr}\\[18pt]
&\ds{+\sqrt{\e}\,\int_0^se^{(s-r)A}\sigma(X_{\varphi, \e}^{t,x}(r),u_\e(t-r,X_{\varphi, \e}^{t,x}(r)))\,dW_r.}
\end{array}
\end{equation}
	
\end{Definition}

\begin{Theorem}
\label{teo2}
	Suppose that Hypotheses \ref{H1} to  \ref{H4} hold, and fix any $g \in\,C^\eta_b(H)$, with $\eta>1/2$, $\e \in\,(0,1)$ and $\delta \in\,[0,\bar{\delta})$, where $\bar{\delta}$ is the constant introduced in Theorem \ref{teo1}.  Moreover, fix an arbitrary predictable process in $L^2(\Omega;C([0,t];H))$ such that 
	\begin{equation}
	\label{sbm10}
	\int_0^t \Vert \varphi(s)\Vert_H^2\,ds\leq M,\ \ \ \ \ \mathbb{P}-\text{a.s.}	
	\end{equation}
	for some $M>0$. Then,  if we assume that $b:H\to H$ is Lipschitz-continuous,  equation \eqref{stoch-pde} admits a  unique mild solution $X^{t, x}_{\varphi, \e} \in\,L^2(\Omega;C([0,t];H))$, for  every $x \in\,H$ and $t>0$. 
	
\end{Theorem}

In what follows, the solution of the uncontrolled version of equation \eqref{stoch-pde}, corresponding to $\varphi=0$, will be denoted by $X^{t, x}_{\e}$.

Once proved Theorem \ref{teo2}, we are interested in studying the limiting behavior of $X^{t, x}_\epsilon$ as $\epsilon\downarrow 0$. More precisely, we want to prove that for every fixed $t>0$ and $x \in\,H$ the family $\{\mathcal{L}(X^{t, x}_\epsilon)\}_{\e \in\,(0,1)}$ satisfies a large deviation principle in the space $C([0,t];H)$ (with speed $\e$) with respect to a suitable action functional $I_{t, x}$ that we will describe explicitly.

In order to state our result, we have to introduce some notations.
First, we introduce the unperturbed  problem
\begin{equation}
\label{sbm5}
Z^\prime(s)=AZ(s)+b(Z(s)),\ \ \ \ Z(0)=y \in\,H.
\end{equation} 
Since we are assuming that $b:H\to H$ is Lipschitz continuous,  for every $T>0$ and $y \in\,H$ there exists a unique $Z^y \in\,C([0,T];H)$ such that
\[Z^y(s)=e^{sA}y+\int_0^s e^{(s-r)A}b(Z^y(r))\,dr.\]
Next, for every $x \in\,H$, $t>0$ and $\varphi \in\,L^2(0,t;H)$ we introduce the controlled problem
\begin{equation}
\label{sbm2}
X^\prime(s)=A	X(s)+b(X(s))+\sigma(X(s),g(Z^{X(s)}(t-s)))\varphi(s),\ \ \ \ X(0)=x.	\end{equation}
In Section \ref{sec8} we will see that under the same assumptions of Theorem \ref{teo2}, equation \eqref{sbm2} admits a unique mild solution $X^{t,x}_\varphi \in\,C([0,t];H)$.
This will allow to state the last main result of this paper.

\begin{Theorem}
\label{teo3}
In addition to the  conditions assumed in Theorem \ref{teo2}, suppose that $g:H\to\mathbb{R}$ is Lipschitz-continuous. Then, for every fixed $t>0$ and $x \in\,H$ the family $\{\mathcal{L}(X^{t, x}_\epsilon)\}_{\e \in\,(0,1)}$ satisfies a large deviation principle in the space $C([0,t];H)$, with speed $\e$, with respect to the action functional
\begin{equation}
\label{sbm1}
I_{t,x}(X)=\frac 12\,\inf\left\{\int_0^t\Vert \varphi(s)\Vert_H^2\,ds\,:\, X(s)=	X^{t,x}_\varphi(s),\ s \in\,[0,t]\right\},
\end{equation}
where $X^{t,x}_\varphi$ is the  unique mild solution of  problem
\eqref{sbm2}.
	
\end{Theorem}

\section{The well-posedness of the stochastic PDE \eqref{stoch-pde}}
\label{sec-spde}

In this section we will assume that, for some $T>0$, $\eta>1/2$, $\vartheta \in\,(0,(\eta-1/2)\wedge 1)$, $\varrho <1/4$, and $\e \in\,(0,1)$ there exists a mild solution $u_\e \in\,C_{\e,\varrho, \eta}((0,T];C^{2+\vartheta}_b(H))$ for equation \eqref{quasi-linear-bis}. We will show how this allows to prove Theorem \ref{teo2} for every $t \in\,(0,T]$.

We  fix $t \in\,(0,T]$ and  a predictable process $\varphi \in\,L^2(\Omega;L^2(0,t;H))$ satisfying \eqref{sbm10} and we consider the stochastic equation
\begin{equation}
\label{stoch-pde-bis}
\left\{\begin{array}{rl}
\ds{dX(s)=} & \ds{\left[A	X(s)+b(X(s))+\sigma(X(s),u_\e(t-s,X(s)))\varphi(s)\right]\,ds}\\[14pt]
& \ds{+\sqrt{\e}\,\sigma(X(s),u_\e(t-s,X(s)))\,dW_s,}\\[10pt]
\ds{X(0)=} & \ds{x.}
	\end{array}\right.
\end{equation}
For every $\e \in\,(0,1)$, $s \in\,[0,t]$ and $x \in\,H$, we define
\begin{equation}
\label{sigma}	
\Sigma_{t,\e}(s,x):=\sigma(x,u_\e(t-s,x)).
\end{equation}
A process $X \in\,L^2(\Omega;C([0,t];H))$ is a mild solution of equation \eqref{stoch-pde-bis} if it is a fixed point of the mapping $\Lambda_{t, \e}$ defined by
\[\begin{array}{ll}	
\ds{\Lambda_{t,\e}(X)(s):=} & \ds{e^{sA}x+\int_0^s e^{(s-r)A}b(X(r))\,ds+\int_0^s e^{(s-r)A}\Sigma_{t,\e}(r,X(r))\,\varphi(r)\,dr}\\[14pt]
&\ds{+\sqrt{\e}\int_0^s e^{(s-r)A}\Sigma_{t,\e}(r,X(r))\,dW(r).}
\end{array}
\]
According to Hypothesis \ref{H3}, for every $\tau>0$, $s \in\,[0,t]$  and $x, y \in\,H$ we have
\[\Vert e^{\tau A} \left(\Sigma_{t,\e}(s,x)-\Sigma_{t,\e}(s,y)\right)\Vert_{\mathcal{L}_2(H)}\leq c\,(\tau\wedge 1)^{-\frac 14}\left(\Vert x-y\Vert_H+|u_\e(t-s,x)-u_\e(t-s,y)|\right).\]
Since $u_\e \in\,C_{\e,\varrho, \eta}((0,T];C^{2+\vartheta}_b(H))$, we have
\[\begin{array}{ll}
\ds{|u_\e(t-s,x)-u_\e(t-s,y)|}  &  \ds{\leq \Vert D_x u_\e(t-s,\cdot)\Vert_0\,\Vert x-y\Vert_H}\\[14pt]
&\ds{\leq \e^{-\varrho}\,((t-s)\wedge 1)^{-\varrho}\Vert u_\e\Vert_{\e, \varrho, \eta, \vartheta, T}\Vert x-y\Vert_H,}	\end{array}\]
so that
\begin{equation}
\label{101}
\begin{array}{l}
\ds{	\Vert e^{\tau A} \left(\Sigma_{t,\e}(s,x)-\Sigma_{t,\e}(s,y)\right)\Vert_{\mathcal{L}_2(H)}\leq c\,(\tau\wedge 1)^{-\frac 14}\left(1+\e^{-\varrho}((t-s)\wedge 1)^{-\varrho}\Vert u_\e\Vert_{\e, \varrho, \eta, \vartheta, T}\right)\Vert x-y\Vert_H.}	
\end{array}
\end{equation}

Now, for every $\beta\geq 0$ we denote by $\mathcal{K}_{\beta, t}(H)$ the 
Banach space of all $H$-valued predictable processes $X$ such that
\[\Vert X\Vert_{\mathcal{K}_{\beta,t}(H)}^2:=\sup_{s \in\,[0,t]}e^{-\beta s}\,\mathbb{E}\, \Vert X(s)\Vert^2_H<\infty.\]
If $X_1, X_2 \in\,\mathcal{K}_{\beta,t}(H)$, in view of \eqref{101} we have
\[\begin{array}{l}
\ds{\mathbb{E}\left\Vert \int_0^s e^{(s-r)A}\left[\Sigma_{t,\e}(r,X_1(r))-\Sigma_{t,\e}(r,X_2(r))\right]\,\varphi(r)\,dr\right\Vert_H^2}\\[18pt]
\ds{\leq \mathbb{E}\int_0^s\left\Vert  e^{(s-r)A}\left[\Sigma_{t,\e}(r,X_1(r))-\Sigma_{t,\e}(r,X_2(r))\right]\right\Vert_{\mathcal{L}_2(H)}^2\,dr\int_0^s\Vert\varphi(r)\Vert_H^2\,dr}\\[18pt]
\ds{\leq c M\,	\mathbb{E}\int_0^s ((s-r)\wedge 1)^{-\frac 12}\left(1+\e^{-2\varrho}((t-r)\wedge 1)^{-2\varrho}\Vert u_\e\Vert^2_{\e, \varrho, \eta, \vartheta, T}\right)\,\Vert X_1(r)-X_2(r)\Vert_H^2\,dr}\\[18pt]
\ds{\leq c M\,\Vert X_1-X_2\Vert^2_{\mathcal{K}_{\beta,t}(H)}\int_0^s ((s-r)\wedge 1)^{-\frac 12}\left(1+e^{-2\varrho}((t-r)\wedge 1)^{-2\varrho}\Vert u_\e\Vert^2_{\e, \varrho, \eta, \vartheta, T}\right)e^{\beta r}\,dr.}
\end{array}\]
Since we are assuming that $\varrho<1/4$, for every $s \in\,[0,t]$
\[\int_0^s ((s-r)\wedge 1)^{-\frac 12}\left(1+e^{-2\varrho}((t-r)\wedge 1)^{-2\varrho}\Vert u_\e\Vert^2_{\e, \varrho, \eta, \vartheta, T}\right)e^{-\beta(s-r)}\,dr\leq c_{\e,\beta, t}(s),\]
for some continuous increasing function $c_{\e,\beta, t}:[0,t]\to [0,+\infty)$ such that 
\[\lim_{\beta\to\infty} \sup_{s \in\,[0,t]}c_{\e,\beta, t}(s)=0.\]
Therefore, we pick $\beta_1=\beta_1(\e,t)>0$ such that 
\[\sup_{s \in\,[0,t]}c_{\e,\beta_1, t}(s)\leq \frac 16,\]
we have
\[\begin{array}{l}
\ds{\sup_{s \in\,[0,t]}e^{-\beta_1 s}\,\mathbb{E}\left\Vert \int_0^s e^{(s-r)A}\left[\Sigma_{t,\e}(r,X_1(r))-\Sigma_{t,\e}(r,X_2(r))\right]\,\varphi(r)\,dr\right\Vert_H^2}\\[18pt]
\ds{\leq \frac 16\,\Vert X_1-X_2\Vert^2_{\mathcal{K}_{\beta_1,t}(H)}.}	\end{array}
\]
Moreover, we have
\[\begin{array}{l}
\ds{\mathbb{E}\left\Vert \int_0^s e^{(s-r)A}\left[\Sigma_{t,\e}(r,X_1(r))-\Sigma_{t,\e}(r,X_2(r))\right]\,dW(r)\right\Vert_H^2}\\[18pt]
\ds{\leq c	\int_0^s ((s-r)\wedge 1)^{-\frac 12}\left(1+\e^{-2\varrho}((t-r)\wedge 1)^{-2\varrho}\Vert u_\e\Vert^2_{\e, \varrho, \eta, \vartheta, T}\right)\,\mathbb{E}\Vert X_1(r)-X_2(r)\Vert_H^2\,dr}\\[18pt]
\ds{\leq c\,\Vert X_1-X_2\Vert^2_{\mathcal{K}_{\beta,t}(H)}\int_0^s ((s-r)\wedge 1)^{-\frac 12}\left(1+\e^{-2\varrho}((t-r)\wedge 1)^{-2\varrho}\Vert u_\e\Vert^2_{\e, \varrho, \eta, \vartheta, T}\right)e^{\beta r}\,dr.}
\end{array}\]
Then, by proceeding as above
\[\begin{array}{l}
\ds{\sup_{s \in\,[0,t]}e^{-\beta_1 s}\,\mathbb{E}\left\Vert \int_0^s e^{(s-r)A}\left[\Sigma_{t,\e}(r,X_1(r))-\Sigma_{t,\e}(r,X_2(r))\right]\,dW(r)\right\Vert_H^2}\\[18pt]
\ds{\leq \frac 16\,\Vert X_1-X_2\Vert^2_{\mathcal{K}_{\beta_1,t}(H)}.}	\end{array}\]
 Finally, due to the Lipschitz-continuity of  $b$, we have that there exists $\beta_2>0$ such that
 that 
 \[\sup_{s \in\,[0,t]}e^{-\beta_2 s}\mathbb{E}\left \Vert \int_0^s e^{(s-r)A}\left[b(X_1(r))-b(X_2(r))\right]\,dr\right\Vert_H^2\leq \frac 16\,\Vert X_1-X_2\Vert^2_{\mathcal{K}_{\beta_2,t}(H)}.\]
Therefore, is we take $\bar{\beta}:=\beta_1\vee \beta_2$ , we have that $\Lambda_{t,\e}$ is a contraction in $\mathcal{K}_{\bar{\beta},t}(H)$ and its fixed point is the unique mild solution  $X^{t,x}_{\varphi,\e}$ of equation \eqref{stoch-pde-bis}. 

Finally, by using a stochastic factorization argument, it is possible to prove that $X_\varphi^{t,x}$ belongs to $L^2(\Omega;C([0,t];H))$ (for all details about stochastic factorization see \cite[Subsection 5.3.1]{DPZ}).

\section{Local existence of mild solutions for   the quasi-linear problem}
\label{sec4}

In this section we will prove  that the quasi-linear problem \eqref{quasi-linear-bis} admits a local mild solution, for every $\e \in\,(0,1)$.

\medskip

In view of \eqref{34} and \eqref{33}, problem \eqref{quasi-linear-bis} can be rewritten as
\begin{equation}
\label{quasi-linear-tris}
\left\{\begin{array}{l}
\ds{D_tu_\e(t,x)=\mathcal{L}_\e u_\e(t,x)+\frac {\e} 2 \text{Tr} \left[\delta\, F(u_\e)(t,x)D^2_x	 u_\e(t,x)\right]+\langle b(x),D u_\e(t,x)\rangle_H,}\\[10pt]
\ds{u_\e(0,x)=g(x),\ \ \ x \in\,H,}
	\end{array}\right.
\end{equation}
where $\mathcal{L}_\e$ is the linear Kolmogorov operator
\[\mathcal{L}_\e\varphi(x)=\frac \e2 \text{Tr}\left[QD_x^2\varphi(x)\right]+\langle Ax, D_x \varphi(x)\rangle_H.\]

As we have recalled in Subsection \ref{ssec2.3}, for every $\varphi \in\,B_b(H)$ the unique classical solution of the linear problem
\[D_tv_\e(t,x)=\mathcal{L}_\e v_\e(t,x),\ \ \ \ v_\e(0,x)=\varphi(x),\] is given by the Ornstein-Uhlenbeck semigroup
\[v_\e(t,x)=R_t^\e\varphi(x)=\int_H \varphi(e^{tA}x+y)\mathcal{N}_{\e Q_t}(dy).\]

Before proceeding with the study of equation \eqref{quasi-linear-tris}, we show how, in view of Hypothesis \ref{H2}, the properties of the Ornstein-Uhlenbeck semigroup described in Subsection \ref{ssec2.3} apply to the semigroup $R^\e_t$.

Thanks to \eqref{67} and \eqref{20}, inequality \eqref{60} gives for every $n \in\,\mathbb{N}\cup \{0\}$ and $\a \in\,(0,1)$
\begin{equation}
\label{main1}
\Vert D^n R^\e_t\varphi\Vert_\a\leq c_{n,\a}\,\e^{-\frac n2}(t\wedge 1)^{-\frac n2} e^{-\la n t}\Vert\varphi\Vert_{t,\a},\ \ \ \ t>0,\ \ \e \in\,(0,1),
\end{equation}
where
\[\Vert\varphi\Vert_{t,\a}:=\left(\Vert \varphi\Vert_0+e^{-\omega \a t}[\varphi]_\a\right).\]
In the same way, inequality \eqref{65} gives for every $n \in\,\mathbb{N}$ and $0\leq \a\leq \beta\leq 1$
\begin{equation}
\label{maineps}
\Vert D^n R^\e_t\varphi\Vert_\a\leq c_{n,\a,\beta}\,\e^{-\frac {n-(\beta-\a)}2}(t\wedge 1)^{-\frac {n-(\beta-\a)}2} e^{-\la n t}\Vert\varphi\Vert_{t,\beta},\ \ \ \ t>0,\ \ \e \in\,(0,1).
	\end{equation}
Finally, since
\[\Vert R^\e_t\varphi\Vert_0\leq \Vert \varphi\Vert_0,\ \ \ \ [R^\e_t\varphi]_\beta \leq e^{-\omega \beta t}[\varphi]_\beta,\ \ \ \ \e>0,\]
and
\[[R^\e_t\varphi]_{\a}\leq \Vert D R^\e_t\varphi\Vert_0^{\frac{\a-\beta}{1-\beta}}[R^\e_t\varphi]_\beta^{\frac{1-\a}{1-\beta}},\]
thanks to \eqref{main1} and \eqref{maineps} we get
\begin{equation}
\label{main3}
[R^\e_t\varphi]_\a\leq c_{\a,\beta}\, \e^{-\frac{\a-\beta}2}(t\wedge 1)^{-\frac{\a-\beta}2} e^{-\omega_{\a,\beta}t}	\Vert\varphi\Vert_{t,\beta},\ \ \ \ t>0,\ \ \e \in\,(0,1),
\end{equation}
for some $\omega_{\a,\beta}>0$, and
\begin{equation}
\label{main2}
\Vert R^\e_t \varphi\Vert_\a\leq c_{\a,\beta}\, \e^{-\frac{\a-\beta}2}(t\wedge 1)^{-\frac{\a-\beta}2} 	\Vert\varphi\Vert_{t,\beta},\ \ \ \ t>0,\ \ \e \in\,(0,1).	
\end{equation}

\medskip

Now, we introduce the notion of mild solution for equation \eqref{quasi-linear-bis}.

\begin{Definition}
A function $u_\e \in\,C(	[0,+\infty);H)$ such that $u_\e(t,\cdot) \in\,C^2_b(H)$, for every $t>0$, is a mild solution for problem \eqref{quasi-linear-bis} if for every $(t,x) \in\,[0,+\infty)\times H$
\[u_\e(t,x)=R^\e_tg(x)+\int_0^t R^\e_{t-s}\left(\frac \e 2\text{\em {Tr}}\left[\delta F(u_\e(s,\cdot))D^2_x	 u_\e(s,\cdot)\right]+\langle b(\cdot),D u_\e(s,\cdot)\rangle_H\right)(x)\,ds.\]
\end{Definition}

For every $R>0$ we define
\[\mathcal{Y}^{\e,R}_{\varrho, \eta, \vartheta, T}:=\left\{\,u \in\,C_{\e,\varrho, \eta}((0,T];C^{2+\vartheta}_b(H))\ :\ \Vert u\Vert_{\e, \varrho, \eta, \vartheta, T}\leq R\right\},\]
and for every $v \in\,\mathcal{Y}^{\e,R}_{ \varrho, \eta, \vartheta, T}$ and $\delta>0$ we define 
\[\Gamma_{\e,\delta}(v)(t,x):=	\int_0^t R^\e_{t-s}\gamma_{\e,\delta}(v,s)(x)\,ds,\ \ \ \ \ t \in\,[0,T],\ \ \ x \in\,H,\]
where
\[\gamma_{\e,\delta}(v,s)(x):=\frac \e2\text{Tr}\left[\delta F(v(s,\cdot))(x)D^2_x	 v(s,x)\right]+\langle b(x),D v(s,x)\rangle_H.\]
In particular, $u_\e$ is a mild solution for problem \eqref{quasi-linear-bis}
 if and only if 
 \[u_\e(t,x)=R^\e_tg(x)+\Gamma_{\e, \delta}(u_\e)(t,x).\]
First, we investigate the dependence of $\gamma_{\e,\delta}$ on $v \in\,\mathcal{Y}^{\e,R}_{\varrho, \eta, \vartheta, T}$.

\begin{Lemma} 
\label{lemma5.1}
For every $v \in\,\mathcal{Y}^{\e,R}_{\varrho, \eta, \vartheta, T}$ and $\delta>0$
\begin{equation}
\label{36}
\Vert \gamma_{\e,\delta}(v,s)\Vert_\vartheta\leq c\,\e^{\frac 12-\varrho}\,\delta R\left(1+R\right)(s\wedge 1)^{-(\varrho+\frac 12)}+c\,\e^{-\varrho}R(s\wedge 1)^{-\varrho},\ \ \ \ \ s \in\,(0,T].	
\end{equation}

Moreover, for every $v_1, v_2 \in\,\mathcal{Y}^{\e,R}_{\varrho, \eta, \vartheta, T}$ and $\delta>0$
\begin{equation}
\label{85}
\begin{array}{l}
\ds{\Vert \gamma_{\e,\delta}(v_1,s)-\gamma_{\e, \delta}(v_2,s)\Vert_\vartheta\leq c\,\,\e^{\frac 12-\varrho}\delta\,R(1+R)\,(s\wedge 1)^{-(\varrho+\frac 12)}\Vert v_1(s,\cdot)-v_2(s,\cdot)\Vert_\vartheta\,}\\[18pt]
\ds{+c\,\e\,\delta(1+R)\,\Vert D^2_x v_1(s,\cdot)-D^2_xv_2(s,\cdot)\Vert_\vartheta+c\,\Vert D_xv_1(s,\cdot)-D_xv_2(s,\cdot)\Vert_\vartheta.}
\end{array}
	\end{equation}

\end{Lemma}

\begin{proof}
In view of  \eqref{25} and Hypothesis \ref{H3}, we have
\[\begin{array}{ll}
\ds{\Vert \gamma_{\e,\delta}(v,s)\Vert_\vartheta} & \ds{\leq c\,\e\,\delta\,\Vert F(v(s,\cdot))\Vert_\vartheta \,\Vert D_x^2v(s,\cdot)\Vert_\vartheta +\Vert b\Vert_\vartheta\,\Vert D_xv(s,\cdot)\Vert_\vartheta}\\[10pt]
& \ds{\leq c\,\e\,\delta\left(1+\Vert v(s,\cdot)\Vert_\vartheta\right)\Vert D^2_xv(s,\cdot)\Vert_\vartheta+c\,\Vert D_xv(s,\cdot)\Vert_\vartheta,}	
\end{array}
\]
and since we are assuming that $v \in\,\mathcal{Y}^{\e,R}_{\varrho, \eta, \vartheta, T}$, this implies \eqref{36}.

Next, if $v_1, v_2 \in\,\mathcal{Y}^{\e,R}_{ \varrho, \eta, \vartheta, T}$ and $\delta>0$ we have
\[\begin{array}{l}
\ds{\Vert \gamma_{\e,\delta}(v_1,s)-\gamma_{\e,\delta}(v_2,s)\Vert_\vartheta\leq c\,\e\,\delta\,\Vert F(v_1(s,\cdot))-F(v_1(s,\cdot))\Vert_\vartheta \,\Vert D_x^2v_1(s,\cdot)\Vert_\vartheta}\\[18pt]
\ds{+c\,\e\,\delta\,\Vert F(v_2(s,\cdot))\Vert_\vartheta \,\Vert D_x^2v_1(s,\cdot)-D_x^2v_2(s,\cdot)\Vert_\vartheta	+\Vert b\Vert_\vartheta\,\Vert D_xv_1(s,\cdot)-D_xv_2(s,\cdot)\Vert_\vartheta.}
\end{array}\]
Thus, according to \eqref{25} and \eqref{28}, 
\[\begin{array}{l}
\ds{\Vert \gamma_{\e,\delta}(v_1,s)-\gamma_{\e,\delta}(v_2,s)\Vert_\vartheta}\\[14pt]
\ds{\leq c\,\e\,\delta\,\Vert v_1(s,\cdot)-v_2(s,\cdot)\Vert_\vartheta\left(1+\Vert v_1(s,\cdot)\Vert_\vartheta+\Vert v_2(s,\cdot)\Vert_\vartheta\right)\Vert D^2_x v_1(s,\cdot)\Vert_\vartheta}	\\[18pt]
\ds{+c\,\e\,\delta \left(1+\Vert v_2(s,\cdot)\Vert_\vartheta\right)\Vert D^2_x v_1(s,\cdot)-D^2_xv_2(s,\cdot)\Vert_\vartheta+c\,\Vert D_xv_1(s,\cdot)-D_xv_2(s,\cdot)\Vert_\vartheta.}
\end{array}\]
Recalling that $v_1, v_2 \in\,\mathcal{Y}^{\e,R}_{\varrho, \eta, \vartheta, T}$, this implies \eqref{85}.

\end{proof}

\begin{Remark}
{\em If for every fixed $\e, \delta>0$ we define
\[\alpha_{\e,\delta}(R,s):=	\e^{\frac 12-\varrho}\delta R(1+R)(s\wedge 1)^{-(\varrho+\frac 12)}+\e^{-\varrho}R(s\wedge 1)^{-\varrho},\ \ \ \ \ s>0,\ \ \ R>0,\]
and
\[\mathfrak{a}_{\e,\delta}(R,s):=	\e^{\frac 12-\varrho}\delta (1+R)^2(s\wedge 1)^{-(\varrho+\frac 12)}+\e^{-\varrho}(s\wedge 1)^{-\varrho},\ \ \ \ \ s>0,\ \ \ R>0,\]
due \eqref{36} and \eqref{85} we have that for every $v, v_1, v_2 \in\,\mathcal{Y}^{\e,R}_{ \varrho, \eta, \vartheta, T}$ and $s \in\,(0,T]$
\begin{equation}
	\label{90}
	\Vert \gamma_{\e,\delta}(v,s)\Vert_\vartheta\leq c\,\alpha_{\e,\delta}(R,s),
\end{equation}
and
\begin{equation}
\label{91}
\Vert \gamma_{\e,\delta}(v_1,s)-\gamma_{\e,\delta}(v_2,s)\Vert_\vartheta\leq c\,\mathfrak{a}_{\e,\delta}(R,s)\,\Vert v_1-v_2\Vert_{\e, \varrho, \eta, \vartheta, T}.
	\end{equation}
	Notice that for all $\beta<1$ and $\mu>0$ and for all $t\geq 0$
	\[\begin{array}{l}
\ds{\int_0^t((t-s)\wedge 1)^{-\beta} e^{-\mu(t-s)}\alpha_{\e,\delta}(R,s)\,ds	}\\[14pt]
 \ds{\leq \e^{\frac 12-\varrho}\delta R(1+R)\int_0^t((t-s)\wedge 1)^{-\beta} e^{-\mu(t-s)}(s\wedge 1)^{-(\varrho+\frac 12)}\,ds}\\[14pt]
 \ds{+\e^{-\varrho}R\int_0^t((t-s)\wedge 1)^{-\beta} e^{-\mu(t-s)}(s\wedge 1)^{-\varrho}\,ds,}
\end{array}\]
and this implies that there exists some constant $c>0$ only dependent on $\beta$ and $\mu$ such that
\begin{equation}
\label{93}
\int_0^t((t-s)\wedge 1)^{-\beta} e^{-\mu(t-s)}\alpha_{\e,\delta}(R,s)\,ds	 \leq c\,(t\wedge 1)^{\frac 12-(\varrho+\beta)}\e^{\frac 12-\varrho}\lambda_{\e,\delta}(T,t),	
\end{equation}
where
\[\lambda_{\e,\delta}(R,t):=\delta R(1+R)\,+\e^{\frac 12}R\,(t\wedge 1)^{\frac 12}.\]
In an analogous way
\begin{equation}
\label{93-bis}
\int_0^t(t-s)^{-\beta} e^{-\mu(t-s)}\mathfrak{a}_{\e,\delta}(R,s)\,ds	 \leq c\,(t\wedge 1)^{\frac 12-(\varrho+\beta)}\e^{\frac 12-\varrho}\mathfrak{l}_{\e,\delta}(R,t),\end{equation}
where
\[\mathfrak{l}_{\e,\delta}(R,t):=\delta (1+R)^2\,+\e^{\frac 12}(t\wedge 1)^{\frac 12}.\]

}
\end{Remark}

Next we prove the following estimates for $\Gamma_{\e,\delta}$ on $\mathcal{Y}^{\e,R}_{ \varrho, \eta, \vartheta, T}$.

\begin{Lemma}
\label{lemma4.4}
For every $v \in\,\mathcal{Y}^{\e,R}_{\e, \varrho, \eta, \vartheta, T}$ and $\e, \delta \in\,(0,1)$ it holds
\begin{equation}
\label{74}
\Vert \Gamma_{\e,\delta}(v)\Vert_{\e, \varrho, \eta, \vartheta, T}\leq c\,\lambda_{\e,\delta}(R,T)\left[\e^{\frac{1-(\eta-\vartheta)}2-\varrho}\,(T\wedge 1)^{\frac{1-(\eta-\vartheta)}2-\varrho}\,(T\vee 1)+1\right].	
\end{equation}
\end{Lemma}

\begin{proof}
{\em Step 1.} We have
\begin{equation}
\label{70}
\Vert \Gamma_{\e,\delta}(v)(t)\Vert_{\eta}\leq 	c\,\e^{\frac{1-(\eta-\vartheta)}2-\varrho}\lambda_{\e,\delta}(R,t)\,(t\wedge 1)^{\frac{1-(\eta-\vartheta)}2-\varrho}(t\vee 1),\ \ \ \ \ t \in\,[0,T].
\end{equation}

{\em Proof of Step 1.} In view of \eqref{main2},  for every $t \in\,[0,T]$
\[\begin{array}{ll}
\ds{	
\Vert \Gamma_{\e,\delta}(v)(t)\Vert_\eta} & \ds{\leq \int_0^t \Vert R^\e_{t-s}\gamma_{\e,\delta}(v,s)\Vert_{\eta}\,ds\leq c\,\e^{-\frac{\eta-\vartheta}2}\int_0^t ((t-s)\wedge 1)^{-\frac{\eta-\vartheta}2}\Vert \gamma_{\e,\delta}(v,s)\Vert_\vartheta\,ds}\\[14pt]
&\ds{\leq c\,\e^{-\frac{\eta-\vartheta}2}\int_0^t((t-s)\wedge 1)^{-\frac{\eta-\vartheta}2}\alpha_{\e,\delta}(R,s)\,ds.}\end{array}
\]
Then, by adapting \eqref{93} to the case $\mu=0$, we get
\[\Vert \Gamma_{\e,\delta}(v)(t)\Vert_\eta\leq  c\,\e^{-\frac{\eta-\vartheta}2}\e^{\frac 12-\varrho}\lambda_{\e,\delta}(R,t)(t\wedge 1)^{\frac{1-(\eta-\vartheta)}2-\varrho}\,(t\vee 1),\]
and \eqref{70} follows.

\medskip

{\em Step 2.} We have
\begin{equation}
\label{71}
(t\wedge 1)^{\varrho} \,\Vert D\Gamma_{\e,\delta}(v)(t)\Vert_\vartheta\leq c\,\e^{-\varrho}\left[\delta\,R(1+R)+\e^{-\frac 12}R (t\wedge 1)^{\frac 12}\right].
\end{equation}

{\em Proof of Step 2.} According to \eqref{main1}, we have that
\[\begin{array}{l}
\ds{\int_0^t \Vert DR^\e_{t-s}\gamma_{\e,\delta}(v,s)\Vert_{\vartheta}\,ds\leq c\e^{-\frac 12}\int_0^t((t-s)\wedge 1)^{-\frac 12}e^{-\la(t-s)}\Vert \gamma_{\e,\delta}(v,s)\Vert_\vartheta\,ds.}
\end{array}\]
Then, thanks to \eqref{90} and \eqref{93} we conclude
\[\begin{array}{l}
\ds{\int_0^t \Vert DR^\e_{t-s}\Gamma_{\e,\delta}(v,s)\Vert_{\vartheta}\,ds\leq c\,\e^{-\frac 12}(t\wedge 1)^{-\varrho}\e^{\frac 12-\varrho}\lambda_{\e,\delta}(R,t),}
\end{array}\]
 and \eqref{71} follows.

\medskip

{\em Step 3.}
We have
\begin{equation}
\label{81}
(t\wedge 1)^{\varrho+\frac 12} \,\Vert D^2\Gamma_{\e,\delta}(v)(t)\Vert_\vartheta\leq c\,\e^{-(\frac 12+\varrho)}\lambda_{\e,\delta}(T,t).
	\end{equation}

{\em Proof of Step 3.} By proceeding as in the proof of Step 2, we have
\[\begin{array}{ll}
\ds{\Vert D^2\Gamma_{\e,\delta}(v)(t)\Vert_0}  &   \ds{\leq c\e^{-1+\frac\vartheta 2}\int_0^t((t-s)\wedge 1)^{-1+\frac \vartheta 2}e^{-\la (t-s)}\alpha_{\e,\delta}(R,s)\,ds}	\\[18pt]
&  \ds{\leq c\,\e^{-1+\frac\vartheta 2}\e^{\frac 12-\varrho}(t\wedge 1)^{-(\varrho+\frac 12)+\frac \vartheta{2}}\lambda_{\e,\delta}(R,t),}
\end{array}\]
and this implies 
\begin{equation}
\label{82}
(t\wedge 1)^{\varrho+\frac 12} \,\Vert D^2\Gamma_{\e,\delta}(v)(t)\Vert_0\leq c\,\e^{-(\frac 12+\varrho)+\frac\vartheta 2}\lambda_{\e,\delta}(R,t)(t\wedge 1)^{\frac \vartheta 2}.	
\end{equation}

Now, for every $x, h \in\,H$ and $t \in\,[0,T]$, we have
\[\begin{array}{l}
\ds{	(t\wedge 1)^{\varrho+\frac 12} \Vert D^2\Gamma_{\e,\delta}(v)(t,x+h)-D^2\Gamma_{\e,\delta}(v)(t,x)\Vert_{\mathcal{L}(H)}\leq c\,\e^{-(\frac 12+\varrho)+\frac\vartheta 2}\lambda_{\e,\delta}(R,t)\,(t\wedge 1)^{\frac \vartheta 2}.}
\end{array}
\]
Hence, if we assume that $\Vert h\Vert_H^2>\e\, t/2$
 we get
\begin{equation}
\label{46}	
\begin{array}{l}
\ds{(t\wedge 1)^{\varrho+\frac 12}\,\Vert D^2\Gamma_{\e,\delta}(v)(t,x+h)-D^2\Gamma_{\e,\delta}(v)(t,x)\Vert_{\mathcal{L}(H)}\leq c\,\e^{-(\frac 12+\varrho)}\lambda_{\e,\delta}(R,t)\,\Vert h\Vert_H^\vartheta.}	
\end{array}
\end{equation}
When $\Vert h\Vert_H^2\leq \e\,t/2$ we have
\[\begin{array}{ll}
\ds{D^2\Gamma_{\e,\delta}(v)	(t,x)}  &  \ds{=\int_0^{\e^{-1}\Vert h\Vert_H^2}D^2 R^\e_s \gamma_{\e,\delta}(v,t-s)(x)\,ds+\int_{\e^{-1}\Vert h\Vert_H^2}^t D^2 R^\e_s \gamma_{\e,\delta}(v,t-s)(x)\,ds}\\[18pt]
&\ds{=:a_{\e,\delta}(h,t,x)+b_{\e,\delta}(h,t,x).}
\end{array}
\]
	Due to \eqref{maineps},  from \eqref{90} we have
\begin{equation}
\label{45}
\begin{array}{l}
\ds{|a_{\e,\delta}(h,t,x+h)-a_{\e,\delta}(h,t,x)|\leq c\,\e^{-1+\frac\vartheta 2}\int_0^{\e^{-1}\Vert h\Vert_H^2}(s\wedge 1)^{-1+\frac \vartheta 2}\Vert \gamma_{\e,\delta}(v,t-s)\Vert_\vartheta\,ds\,}\\[18pt]
\ds{\leq c\,\e^{-1+\frac\vartheta 2}\int_0^{\e^{-1}\Vert h\Vert_H^2}s^{-1+\frac \vartheta 2}e^{-\la s}\alpha_{\e,\delta}(R,s)\,ds\leq c\,\e^{-(\frac 12+\varrho)}\lambda_{\e,\delta}(R,t)\,(t\wedge 1)^{-(\varrho+\frac 12)}\,\Vert h\Vert_H^\vartheta.}
\end{array}
\end{equation}
As for $b_{\e,\delta}(h, t,\cdot)$, we have
\[\begin{array}{l}
\ds{b_{\e,\delta}(h,t,x+h)-b_{\e,\delta}(h,t,x)=\int_{\e^{-1}\Vert h\Vert_H^2}^t \left[D^2 R^\e_s \gamma_{\e,\delta}(v,t-s)(x+h)-D^2 R^\e_s \gamma_{\e,\delta}(v,t-s)(x)\right]\,ds.}	
\end{array}\]
Hence, due again to \eqref{maineps}, we have
\[\begin{array}{l}
\ds{\Vert b_{\e,\delta}(h,t,x+h)-b_{\e,\delta}(h,t,x)\Vert_{\mathcal{L}(H)}}\\[10pt]
\ds{\leq\int_{\e^{-1}\Vert h\Vert_H^2}^t \Vert D^2 R^\e_s \gamma_{\e,\delta}(v,t-s)(x+h)-D^2 R^\e_s \gamma_{\e,\delta}(v,t-s)(x)\Vert_{\mathcal{L}(H)}\,ds}\\[18pt]
\ds{\leq c\,	\int_{\e^{-1}\Vert h\Vert_H^2}^t\Vert R^\e_s \gamma_{\e,\delta}(v,t-s)\Vert_{3}\,ds\Vert h\Vert_H}\\[18pt]
\ds{\leq c\,\e^{-\frac{3-\vartheta}2}\e^{\frac 12-\varrho}\int_{\e^{-1}\Vert h\Vert_H^2}^t (s\wedge 1)^{-\frac{3-\vartheta}2}e^{-2\la s}((t-s)\wedge 1)^{-(\varrho+\frac 12)}\lambda_{\e,\delta}(R,t-s)\,ds.}
\end{array}\]
Since we assuming $\Vert h\Vert_H^2\leq \e\,t/2$, we have
\[\begin{array}{l}
\ds{\int_{\e^{-1}\Vert h\Vert_H^2}^t (s\wedge 1)^{-\frac{3-\vartheta}2}e^{-2\la s}((t-s)\wedge 1)^{-(\varrho+\frac 12)})\,ds}\\[18pt]
\ds{=\int_{\e^{-1}\Vert h\Vert_H^2}^{t/2} (s\wedge 1)^{-\frac{3-\vartheta}2}e^{-2\la s}((t-s)\wedge 1)^{-(\varrho+\frac 12)}\,ds+\int_{t/2}^t s^{-\frac{3-\vartheta}2}e^{-2\la s}((t-s)\wedge 1)^{-(\varrho+\frac 12)}\,ds}\\[18pt]
\ds{\leq c\,t^{-(\varrho+\frac 12)}\,\Vert h\Vert_H^{-1+\vartheta}+c\,t^{-\frac{3-\vartheta}2}\,t^{\frac12-\varrho}=c\,t^{-\varrho}\left(\e^{\frac{1-\vartheta}2}\Vert h\Vert_H^{-1+\vartheta}+t^{-\frac{1-\vartheta}2}\right)}\\[14pt]
\ds{\leq c\,\e^{\frac{1-\vartheta}2}(t\wedge 1)^{-(\varrho+\frac 12) }\,\Vert h\Vert_H^{-1+\vartheta}.}	
\end{array}\]
Moreover, in the same way we have
\[\int_{\e^{-1}\Vert h\Vert_H^2}^t s^{-\frac{3-\vartheta}2}e^{-2\la s}((t-s)\wedge 1)^{-\varrho}\,ds\leq c\,\e^{\frac{1-\vartheta}2}(t\wedge 1)^{-\varrho }\,\Vert h\Vert_H^{-1+\vartheta},\]
so that
\[(t\wedge 1)^{\varrho+\frac 12}\,\Vert b_{\e,\delta}(h,t,x+h)-b_{\e,\delta}(h,t,x)\Vert_{\mathcal{L}(H)}\leq c\,\e^{-(\frac 12+\varrho)}\lambda_{\e,\delta}(R,t)\Vert h\Vert_H^\vartheta.\]
This, together with \eqref{45} and \eqref{46}, implies that for every $h \in\,H$
\[\begin{array}{l}
\ds{(t\wedge 1)^{\varrho+\frac 12} \Vert D^2 \Gamma_{\e,\delta}(v)(t,x+h)-D^2\Gamma_{\e,\delta}(v)(t,x)\Vert_{\mathcal{L}(H)}\leq c\,\e^{-(\frac 12+\varrho)}\lambda_{\e,\delta}(R,t)\,\Vert h\Vert_H^\vartheta.}	
\end{array}
\]
Thus, thanks to \eqref{82}, we obtain \eqref{81}.

\medskip

{\em Conclusion.} Estimate \eqref{74} is a consequence of \eqref{70}, \eqref{71} and \eqref{81}.
\end{proof}

\begin{Remark}
\label{remark5.5}
{\em From the proof 	of the previous lemma, we easily see that for every $t \in\,(0,T]$ and $\e \in\,(0,1)$
\begin{equation}
\label{75}
\begin{array}{l}
\ds{\e^{\frac{\eta-\vartheta}2}[\Gamma_{\e,\delta}(v)(t)]_\eta+\e^{\varrho}(t\wedge 1)^\varrho\Vert D\Gamma_{\e,\delta}(v)(t)\Vert_\vartheta+\e^{\frac 12+\varrho}(t\wedge 1)^{\varrho+\frac 12}\Vert D^2\Gamma_{\e,\delta}(v)(t)\Vert_\vartheta}\\[18pt]
\ds{\leq c\,\lambda_{\e,\delta}(R,t)\left((t\wedge 1)^{\frac {1-(\eta-\vartheta)}2-\varrho}\e^{\frac {1-(\eta-\vartheta)}2-\varrho}+1\right),}
\end{array}
	\end{equation}
for some constant $c>0$ independent of $T>0$. Actually, 
in view of \eqref{main3} and \eqref{90}, we have
\[\begin{array}{ll}
\ds{[\Gamma_{\e,\delta}(v)(t)]_\eta}  &  \ds{\leq c\,\e^{-\frac{\eta-\vartheta}2}\int_0^t e^{-\omega_{\vartheta,\eta} (t-s)}((t-s)\wedge 1)^{-\frac{\eta-\vartheta}2}[\gamma_{\e,\delta}(v,s)]_\vartheta\,ds}\\[14pt]
&\ds{\leq c\,\e^{-\frac{\eta-\vartheta}2}\int_0^t e^{-\omega_{\vartheta,\eta} (t-s)}((t-s)\wedge 1)^{-\frac{\eta-\vartheta}2}\alpha_{\e,\delta}(R,s)\,ds}\\[18pt]
& \ds{\leq 	c\,\lambda_{\e,\delta}(R,t)\left(\e^{\frac {1-(\eta-\vartheta)}2-\varrho}(t\wedge 1)^{\frac {1-(\eta-\vartheta)}2-\varrho}+1\right).}	
\end{array}
\]
This, together with \eqref{71} and \eqref{81},  implies \eqref{75}.}
\end{Remark}

\medskip

Now we are ready to prove the existence of a local mild solution.

\begin{Theorem}
\label{lemma1}
Fix $\eta>1/2$ and $\vartheta \in\,(0,(\eta-1/2)\wedge 1)$ and define
\begin{equation}
\label{rho}
\varrho:=\frac{1-(\eta-\vartheta)}2.	\end{equation}
Then, there exist $\delta_1, T_1>0$ such that for every $\e \in\,(0,1)$ problem \eqref{quasi-linear-bis} has a mild solution $u_\e$ in $C_{\e,\varrho, \eta}((0,T_1];C^{2+\vartheta}_b(H))$, for every $\delta\leq \delta_1$.
\end{Theorem}
\begin{proof}
A function $u_\e$ is a mild solution of equation \eqref{quasi-linear-bis} if and only if it is a fixed point for the mapping $\Gamma_{\e,\delta}^g$ defined by
\[\Gamma_{\e,\delta}^g(v)(t)=R^\e_t g+\Gamma_{\e,\delta}(v)(t)      ,\ \ \ \ t \in\,[0,T].\]
Thus, we will prove the existence of a local mild solution for equation \eqref{quasi-linear-bis} by showing that there exist some $T_1, R>0$ and $\delta_1>0$ such that $\Gamma_{\e,\delta}^g$ maps $\mathcal{Y}^{\e,R}_{\varrho, \vartheta, T_1}$ into itself as a contraction, for every $\delta\leq \delta_1$. 

Thanks to \eqref{main2}  we have
\begin{equation}
	\label{83}
	\Vert R^\e_t g\Vert_\eta\leq c\,\Vert g\Vert_\eta.
\end{equation}
Moreover, thanks  to \eqref{maineps}
\begin{equation}
\label{84}
\begin{array}{l}
\ds{\Vert DR^\e_tg(v)(t)\Vert_\vartheta\leq c\,\e^{-\frac{1-(\eta-\vartheta)}2}(t\wedge 1)^{-\frac{1-(\eta-\vartheta)}2}\,e^{-\lambda t}\,\Vert g\Vert_\eta,}\\[10pt]
\ds{\Vert D^2R^\e_tg(v)(t)\Vert	_\vartheta\leq 	c\,\e^{-(1-\frac{\eta-\vartheta}2)}(t\wedge 1)^{-(1-\frac{\eta-\vartheta}2)}\,e^{-2\lambda t}\,\Vert g\Vert_\eta.}	
\end{array}
\end{equation}
Therefore, if we define $\varrho$ as in \eqref{rho}, we have that $\varrho \in\,(0,1/4)$ and from \eqref{83} and \eqref{84} it follows
\begin{equation}
\label{71-bis}
\Vert R^\e_t g\Vert_\eta+\e^{\varrho}(t\wedge 1)^{\varrho} \Vert DR^\e_tg(t)\Vert_\vartheta+\e^{\frac 12+\varrho}(t\wedge 1)^{\varrho+\frac 12}\Vert D^2R^\e_tg(t)\Vert	_\vartheta\leq c\,\Vert g\Vert_\eta.
\end{equation}
With $\varrho$ defined as in \eqref{rho}, together with \eqref{74} this implies that 
for every $v \in\,\mathcal{Y}^{\e,R}_{\varrho, \eta, \vartheta, T}$
\[\Vert \Gamma_{\e, \delta}^g(v)\Vert_{\e, \varrho, \eta, \vartheta, T}\leq c\,\Vert g\Vert_\eta+c\,\left[\delta\,R(1+R)+\e^{-\frac 12}R\,(T\wedge 1)^{\frac 12}\right](T\vee 1).\]
In particular, if we first take $R:=3 c\,\Vert g\Vert_\eta$ and $\delta^\prime>0$ small enough such that
\[c\,\delta^\prime\,R(1+R)\leq \frac R3, \]
and then fix $T^\prime\leq 1$ small enough so that
\[\e^{-\frac 12}R(T^\prime\wedge 1)^{\frac 12}\leq \frac R3,\]
we conclude that for every $\delta\leq \delta^\prime$ and $T\leq T^\prime$
\[\Vert \Gamma_{\e,\delta}^g(v)\Vert_{\e, \varrho, \eta, \vartheta, T}\leq R,\]
so that $\Gamma_{\e,\delta}^g$ maps $\mathcal{Y}^{\e,R}_{\varrho, \eta, \vartheta, T}$ into itself. 

Now, if we fix $v_1, v_2 \in\,\mathcal{Y}^{\e,R}_{\varrho, \eta, \vartheta, T}$,  we have 
\[\begin{array}{l}
\ds{\Vert	\Gamma_{\e,\delta}^g(v_1)-\Gamma_{\e,\delta}^g(v_2)\Vert_{\e, \varrho, \eta, \vartheta, T}\leq c\int_0^T (t-s)^{-\frac{\eta-\vartheta}2}\Vert \gamma_{\e,\delta}(v_1,s)-\gamma_{\e,\delta}(v_2,s)\Vert_\vartheta\,ds}\\[18pt]
\ds{+\e^\varrho\sup_{t \in\,(0,T]}(t\wedge 1)^\varrho\int_0^t\e^{-\frac 12}((t-s)\wedge 1)^{-\frac 12}e^{-\la(t-s)}\Vert \gamma_{\e,\delta}(v_1,s)-\gamma_{\e,\delta}(v_2,s)\Vert_\vartheta\,ds}\\[18pt]
\ds{+\e^{\frac 12+\varrho}\sup_{t \in\,(0,T]}(t\wedge 1)^{\varrho+\frac 12}\int_0^t\e^{-(1-\frac\vartheta2)}(t-s)^{-(1-\frac \vartheta2)}e^{-2\lambda (t-s)}\Vert \gamma_{\e,\delta}(v_1,s)-\gamma_{\e,\delta}(v_2,s)\Vert_\vartheta\,ds}\\
[18pt]
\ds{+\e^{\frac 12+\varrho}\sup_{t \in\,(0,T]}(t\wedge 1)^{\varrho+\frac 12}\,[D^2\Gamma_{\e,\delta}^g(v_1)(t)-D^2\Gamma_{\e,\delta}^g(v_2)(t)]_\vartheta=:\sum_{i=1}^4I_{\delta,i}(\e).}
\end{array}\]
Then, according to  \eqref{91} and \eqref{93}, we have
\begin{equation}
\label{86-bis}
\begin{array}{l}
\ds{I_{\delta, 1}(\e)\leq c\,\mathfrak{l}_{\e,\delta}(R,T)\,\Vert v_1-v_2\Vert_{\e, \varrho, \eta, \vartheta, T}.}
\end{array}	
\end{equation}
In the same way,
\begin{equation}
\label{87}	
\begin{array}{ll}
\ds{I_{\delta,2}(\e)}  &  \ds{\leq c\,\e^{\varrho}\sup_{t \in\,(0,T]}(t\wedge 1)^{\varrho}\int_0^t\e^{-\frac 12}(t-s)^{-\frac 12}e^{-\lambda(t-s)}\mathfrak{a}_{\e,\delta}(R,s)\,ds\,\Vert v_1-v_2\Vert_{\e, \varrho, \eta, \vartheta, T}}	\\[18pt]
& \ds{\leq c\,\mathfrak{l}_{\e,\delta}(R,T)\Vert v_1-v_2\Vert_{\e, \varrho, \eta, \vartheta, T}.}
\end{array}\end{equation}
and 
\begin{equation}
\label{86}	
\begin{array}{l}
\ds{I_{\delta,3}(\e)}\\[14pt]
\ds{\leq c\,\e^{\frac 12+\varrho}\sup_{t \in\,(0,T]}(t\wedge 1)^{\varrho+\frac 12}\int_0^t\e^{-(1-\frac \vartheta2)}((t-s)\wedge 1)^{-(1-\frac \vartheta2)}e^{-2\lambda(t-s)}\mathfrak{a}_{\e,\delta}(R,s)\,ds\,\Vert v_1-v_2\Vert_{\e, \varrho, \eta, \vartheta, T}}	\\[18pt]
\ds{\leq c\,\e^{\frac\vartheta 2}(T\wedge 1)^{\frac \vartheta2}\,\mathfrak{l}_{\e,\delta}(R,T)\,\Vert v_1-v_2\Vert_{\e, \varrho, \eta, \vartheta, T}.}
\end{array}\end{equation}
As for $I_{\delta, 4}(\e)$, due to \eqref{86},  if we fix any $x, h \in\,H$ and assume  $\Vert h\Vert_H^2>\e\, t/2$  we have
 \begin{equation}
 \label{96}
 \begin{array}{l}
\ds{\e^{\varrho+\frac 12}(t\wedge 1)^{\varrho+\frac 12}\,\Vert D^2\Gamma_{\e,\delta}^g(v_1)(t,x+h)-D^2\Gamma_{\e,\delta}^g(v_2)(t,x)\Vert_{\mathcal{L}(H)}}\\[18pt]
\ds{\leq c\,\e^{\frac \vartheta2}(t\wedge 1)^{\frac \vartheta2}\mathfrak{l}_{\e,\delta}(R,t)\Vert v_1-v_2\Vert_{\e, \varrho, \eta, \vartheta, T}\leq c\, \mathfrak{l}_{\e,\delta}(R,t)\,\Vert v_1-v_2\Vert_{\e, \varrho, \eta, \vartheta, T}\Vert h\Vert_H^\vartheta.}\end{array}	
 \end{equation}
On the other hand, if we assume that $\Vert h\Vert_H^2\leq \e\,t/2$ we write
\[D^2\Gamma_{\e,\delta}^g(v_1)(t,x)-D^2\Gamma_{\e,\delta}^g(v_2)(t,x)=a_{\e,\delta}(h,t,x)+b_{\e,\delta}(h,t,x),\]
where
\[a_{\e,\delta}(h,t,x):=\int_0^{\e^{-1}\Vert h\Vert_H^2} D^2R^\e_s\left(\gamma_{\e,\delta}(v_1,t-s)-\gamma_{\e,\delta}(v_2,t-s)\right)\,ds,\]
and
\[b_{\e,\delta}(h,t,x)=:\int_{\e^{-1}\Vert h\Vert_H^2}^t D^2R^\e_s\left(\gamma_{\e,\delta}(v_1,t-s)-\gamma_{\e,\delta}(v_2,t-s)\right)\,ds.\]
Then, thanks to \eqref{91} and \eqref{93}, we can proceed as in Step 3 of the proof of Lemma \ref{lemma4.4} and we obtain that \eqref{96} holds also when $\Vert h\Vert_H^2\leq t/2$. In particular, we obtain that
\[\e^{\varrho+\frac 12}(t\wedge 1)^{\varrho+\frac 12}\,[D^2\Gamma_{\e,\delta}^g(v_1)(t)-D^2\Gamma_{\e,\delta}^g(v_2)(t)]_\vartheta\leq c \,\mathfrak{l}_{\e,\delta}(R,t)\Vert v_1-v_2\Vert_{\e, \varrho, \eta, \vartheta, T},\]
so that
\begin{equation}
\label{97}
I_{\delta,4}(\e)\leq c \,\mathfrak{l}_{\e,\delta}(R,t)\Vert v_1-v_2\Vert_{\e, \varrho, \eta, \vartheta, T}.	
\end{equation}

Therefore, if we combine \eqref{86-bis}, \eqref{87}, \eqref{86} and \eqref{97}, we obtain that 
\[\Vert	\Gamma_{\e,\delta}^g(v_1)-\Gamma_{\e,\delta}^g(v_2)\Vert_{\e, \varrho, \eta, \vartheta, T}\leq c\, \mathfrak{l}_{\e,\delta}(R,t)\Vert v_1-v_2\Vert_{\e, \varrho, \eta, \vartheta, T}.\]
This means that if we first  choose $\delta_1\leq \delta^\prime$ such that
\[ c\,\delta_1\,(1+R)^2<\frac 12,\]
and then
$T_1\leq T^\prime$ such that
\[ c\,\e^{-\frac 12}\,(T_1\wedge 1)^{\frac 12}<\frac 12,\]
we can conclude that $\Gamma_{\delta, g}$ maps $\mathcal{Y}^R_{\varrho, \vartheta, T_1}$ into itself as a contraction, for every $\delta\leq \delta_1.$
\end{proof}

\section{Further  properties of mild solutions of the quasi-linear problem}
\label{ssec4.2}
We will show that any mild solution $u_\e$  of equation \eqref{quasi-linear-bis} that belongs to $C_{\e,\varrho, \eta}((0,T];C^{2+\vartheta}_b(H))$ is in fact a classical solution in the sense of Definition \ref{def-classical}. Moreover, by using its probabilistic interpretation   in terms of equation \eqref{stoch-pde}, we will prove that a maximum principle holds for equation \eqref{quasi-linear-bis}. This will imply that the local mild solution we have found in Section \ref{sec4} is the unique global classical solution of Theorem \ref{teo1}.

\medskip

We start by proving that $QD^2_xu_\e(t,x)$ is a trace-class operator. 
\begin{Lemma}
\label{lemma4.6}
For every $t \in\,(0,T]$ and $x \in\,H$, we have that $QD^2_x u_\e(t,x) \in\,\mathcal{L}_1(H)$.
\end{Lemma}

\begin{proof}
If $u_\e$ is a mild solution, with the notations  we have introduced in Section \ref{sec4} we have
\[u_\e(t,x)=R^\e_t g(x)+\Gamma_{\e,\delta}(u_\e)(t,x).\]
According to \eqref{16} we have that $Q D^2_x R^\e_t g(t,x) \in\,\mathcal{L}_1(H)$, for every $\e \in\,(0,1)$, $t>0$ and $x \in\,H$, and
thanks to \eqref{16} and \eqref{100}
\[\Vert Q D^2_x R^\e_t g(t)\Vert_0\leq \kappa_{\e,\eta}(t/2)\,\Vert g\Vert_\eta.\] 
As far as $\Gamma_{\e,\delta}(u)$ is concerned, if $R=\Vert u_\e\Vert_{\e, \varrho, \eta, \vartheta, T}$, thanks  to \eqref{16}, \eqref{100} and \eqref{90}, we have
\[\begin{array}{l}
\ds{\Vert QD^2\Gamma_{\e,\delta}(u)(t,x)\Vert_{\mathcal{L}_1(H)}\leq \int_0^t \Vert QD^2_xR^\e_{t-s}\gamma_{\e,\delta}(u,s)(x)\Vert_{\mathcal{L}_1(H)}\,ds\leq \int_0^t\kappa_{\e,\vartheta}(t-s)\Vert \gamma_{\e,\delta}(s)\Vert_\vartheta\,ds}\\[18pt]
\ds{\leq c\,e^{-\beta_\vartheta}\int_0^t((t-s)\wedge 1)^{-\beta_\vartheta} e^{-\alpha_\vartheta\frac{(t-s)}2 }\e^{-(\varrho+\frac 12)}(s\wedge 1)^{-(\frac 1+\varrho)}\lambda_{\e,\delta}(R,s)\,ds\leq c_\e(R,T).}	\end{array}\]
This allows to conclude that $QD^2_x u_\e(t,x) \in\,\mathcal{L}_1(H)$ for every $\e \in\,(0,1)$, $t \in\,(0,T]$ and $x \in\,H$.

\end{proof}

Next, we show that $u_\e$ is differentiable with respect to $t \in\,(0,T]$ and $x \in\,D(A)$ and is a classical solution of equation \eqref{quasi-linear-bis}. In Subsection \ref{ssec2.3}, we have seen that for every $\varphi \in\,B_b(H)$ and $x \in\,D(A)$ the mapping
\[t \in\,(0,+\infty)\mapsto R^\e_t\varphi(x) \in\,\mathbb{R},\]
is differentiable and 
\[D_t R^\e_t\varphi(x)=\mathcal{L}_\e R^\e_t\varphi(x).\]
Hence,  since 
\[\begin{array}{ll}
\ds{u_\e(t,x)} & \ds{=R^\e_tg(x)+\int_0^t R^\e_{t-s}\left(\frac \e 2\text{Tr}\,[\delta\,F(u_\e(s,\cdot))D^2_x u_\e(s,\cdot)]+\langle b(\cdot),Du_\e(s,\cdot)\rangle_H\right)(x)\,ds}\\[18pt]
& \ds{=R^\e_tg(x)+\Gamma_{\e,\delta}(u_\e)(t,x),}	
\end{array}
\]
thanks to Lemma \ref{lemma4.6}, for every $x \in\,D(A)$  we can differentiate both sides with respect to $t>0$, and we get
\[\begin{array}{ll}
\ds{D_t u_\e(t,x)} & \ds{=\mathcal{L}_\e R^\e_tg(x)+\frac \e 2\text{Tr}\,[\delta\,F(u_\e(t,x))D^2_x u_\e(t,x)]+\langle b(x),Du_\e(t,x)\rangle_H+\mathcal{L}_\e\Gamma_{\e,\delta}(t)(t,x)}\\[18pt]
& \ds{=\mathcal{L}_\e u_\e(t,x)+\frac \e 2\text{Tr}\,[\delta\,F(u_\e(t,x))D^2_x u_\e(t,x)]+\langle b(x),Du_\e(t,x)\rangle_H.}	
\end{array}
\]

Thus, we have proven the following result.
\begin{Theorem}
Under Hypotheses \ref{H1} to \ref{H4}, if $u_\e$ is a mild solution 
 of equation \eqref{quasi-linear-bis} that belongs to $C_{\e,\varrho, \eta}((0,T];C^{2+\vartheta}_b(H))$,  then it is a classical solution.	\end{Theorem}

Next, we show how  any solution 
 of equation \eqref{quasi-linear-bis} is related to the stochastic PDE \eqref{stoch-pde}.
 
 \begin{Theorem}
 	Assume Hypotheses \ref{H1} to \ref{H4}. Then if $u_\e\in\,C_{\e,\varrho, \eta}((0,T];C^{2+\vartheta}_b(H))$ is a solution of equation \eqref{quasi-linear-bis}  and $X_\e^{t,x} \in\,L^2(\Omega;C([0,t];H))$ is a solution of equation   \eqref{stoch-pde}, we have
 	\begin{equation}
 	\label{106}
 	u_\e(t,x)=\mathbb{E} g(X_\e^{t,x}(t)).	
 	\end{equation}
 \end{Theorem}
 
 \begin{proof}
 	The natural way  to prove \eqref{106} is by applying the It\^o formula to the function $(s,x) \in\,[0,t]\times H\mapsto u_\e(t-s,x)$ and to the process $X_\e^{t,x}(s)$. However, we cannot do this directly first because $u_\e$ satisfies equation  \eqref{quasi-linear-bis} in classical sense only for $x \in\,D(A)$ and second because $X_\e^{t,x}$ is only a mild solution of equation \eqref{stoch-pde}, and not a strong solution, as required when It\^o's formula is used. To overcome these difficulties, we introduce a suitable approximation of $u_\e$ and $X_\e^{t,x}$, by adapting an argument introduced in \cite[Proof of Theorem 9.25]{DPZ}.
 	 
 For every $m \in\,\mathbb{N}$ we define
 $J_m =m(m-A)^{-1}$ and
 \[u_{\e,m}(t,x)=u_\e(t,J_mx),\ \ \ \ \ \ (t,x) \in\,[0,T]\times H.\]
 Since $J_mx \to x$, as $m\to\infty$, and $u_\e\in\,C_{\e,\varrho, \eta}((0,T];C^{2+\vartheta}_b(H))$,
 we have that
\begin{equation}
	\label{107}
	\lim_{m\to \infty}\sup_{t \in\,[0,T]}\Vert u_{\e,m}(t,\cdot)-u(t,\cdot)\Vert_0=0.
\end{equation} 
Moreover, 
\begin{equation}
\label{110}
D_xu_{\e,m}(t,x)=	J_m^\star D_x u(t,J_m x),\ \ \ \ D^2_x u_{\e,m}(t,x)=J_m^\star D^2_x u(t,J_m x)J_m.
\end{equation}

Next, for every $m \in\,\mathbb{N}$ we introduce the stochastic PDE
\begin{equation}
\label{108}
\left\{\begin{array}{l}
\ds{dX_{\e,m}^{t,x}(s)=\left[A	X_{\e,m}^{t,x}(s)+J_m b(X_{\e,m}^{t,x}(s))\right]\,ds+J_m\Sigma_{\e,t}(s,X_{\e,m}^{t,x}(s))\, dW^m_s,}\\[10pt]
\ds{X^{t,x}(0)=J_mx,}
	\end{array}\right.
\end{equation}
where $\Sigma_{\e,t}$ is the operator introduced in \eqref{sigma} and
 $W^m_t$ is the projection of the cylindrical Wiener process $W_t$ onto $H_m:=\text{span}\{e_1,\ldots,e_m\}$. 
By proceeding as in Section \ref{sec-spde}, we can prove that equation \eqref{108} admits  a unique mild solution $X^{t,x}_{\e,m}\in\,L^2(\Omega;C([0,T];H))$. Since $J_m$ maps $H$ into $D(A)$ and $W^m_t$ is a finite dimensional noise, it is immediate to check that 	$X^{t,x}_{\e,m}$ is a strong solution. Namely
\[X^{t,x}_{\e,m}(s)=J_mx+\int_0^s AX^{t,x}_{\e,m}(r)\,dr+ \int_0^s J_mb(X_{\e,m}^{t,x}(r))\,dr+\int_0^s J_m\Sigma_{\e,t}(r,X_{\e,m}^{t,x}(r))\,dW^m(r).\]
 At the end of this section we will prove that
 \begin{equation}
 \label{109}
 \lim_{m\to\infty}\sup_{s \in\,[0,t]}\mathbb{E}\Vert X^{t,x}_{\e,m}(s)-X^{t,x}(s) \Vert_H=0.	
 \end{equation}
	
 Now we apply It\^o's formula to  $u_{\e,m}$ 	and $X^{t,x}_{\e,m}$ and thanks to \eqref{110} we get
 \[\begin{array}{l}
\ds{d_su_{\e,m}(t-s,X^{t,x}_{\e,m}(s))=-D_t u(t-s,J_mX^{t,x}_{\e,m}(s))\,ds}\\[10pt]
\ds{+\frac 12 \text{Tr}\left[J_m^\star D^2_x u(t-s,J_mX^{t,x}_{\e,m}(s))J_m(J_m\Sigma_{\e,t}(r,X_{\e,m}^{t,x}(r)))(J_m\Sigma_{\e,t}(r,X_{\e,m}^{t,x}(r)))^\star\right]\,ds}	\\[18pt]
\ds{+\langle AJ_mX^{t,x}_{\e,m}(s)+ J_m^2b(X_{\e,m}^{t,x}(s)),D_xu(t-s,J_mX^{t,x}_{\e,m}(s))\rangle_H\,ds}\\[18pt]
\ds{+\langle J^2_m\Sigma_{\e,t}(r,X_{\e,m}^{t,x}(s))\,dW^m(s),D_xu(t-s,J_mX^{t,x}_{\e,m}(s))\rangle_H.}
\end{array}
\]
Therefore, recalling that $u(t,x)$ satisfies equation \eqref{quasi-linear-tris}, for every $x \in\,D(A)$, since $J_mX^{t,x}_{\e,m}(s) \in\,D(A)$ we have
 \begin{equation}
 \label{111}	
 \begin{array}{l}
 \ds{d_su_{\e,m}(t-s,X^{t,x}_{\e,m}(s))}\\[18pt]
 \ds{=\langle J^2_m\Sigma_{\e,t}(r,X_{\e,m}^{t,x}(s))\,dW^m(s),D_xu(t-s,J_mX^{t,x}_{\e,m}(s))\rangle_H+[I_{m,1}^\e(s)+I_{m,2}(s)^\e]\,ds,}
 	\end{array}\end{equation}
where
\[\begin{array}{l}
\ds{I_{m,1}^\e(s):=\frac \e2 \text{Tr}\left[J_m^\star D^2_x u(t-s,J_mX^{t,x}_{\e,m}(s))J_m(J_m\Sigma_{\e,t}(r,X_{\e,m}^{t,x}(r)))(J_m\Sigma_{\e,t}(r,X_{\e,m}^{t,x}(r)))^\star\right]}	\\[18pt]
\ds{-\frac \e2 \text{Tr}\left[D^2_x u(t-s,J_mX^{t,x}_{\e,m}(s))(\Sigma_{\e,t}(r,J_mX_{\e,m}^{t,x}(r)))(\Sigma_{\e,t}(r,J_mX_{\e,m}^{t,x}(r)))^\star\right],}
\end{array}
\]
 and
\[I_{m,2}^\e(s):=\begin{array}{l}
\ds{\langle  J_m^2b(X_{\e,m}^{t,x}(s))-b(J_mX_{\e,m}^{t,x}(s)),D_xu(t-s,J_mX^{t,x}_{\e,m}(s))\rangle_H.}
\end{array}
\]
If we take the expectation of both sides in \eqref{111} and integrate with respect to $s \in\,[0,t]$  we get
\begin{equation}
\label{112}	
\mathbb{E} g(J_mX_{\e,m}^{t,x}(t))=u_{\e,m}(t,J_m x)+\int_0^t \mathbb{E}\left( I_{m,1}^\e(s)+I_{m,2}^\e(s)\right)\,ds.
\end{equation}
In view of \eqref{107} and \eqref{109}, we have that 
\[\lim_{m\to\infty }\mathbb{E} g(J_mX_{\e,m}^{t,x}(t))=\mathbb{E} g(X^{t,x}(t)),\ \ \ \ \ \ \ \ \ \lim_{m\to\infty}u_{\e,m}(t,J_m x)=u(t,x).\]
Moreover, since $u_\e \in\,C_{\e,\varrho, \eta}((0,T];C^{2+\vartheta}_b(H))$, by using again \eqref{107} and \eqref{109} it is not difficult to check that
\[\lim_{m\to\infty} \int_0^t \mathbb{E}\left( |I_{m,1}^\e(s)|+|I^\e_{m,2}(s)|\right)\,ds=0.\]
Therefore, if we take the limit  of both sides in \eqref{112}, as $m\to\infty$, we obtain \eqref{106}.

 \end{proof}
 
 \begin{Remark}
 	{\em Thanks to  the representation formula \eqref{106} of $u_\e$, we have that 
 	\begin{equation}
 	\label{113}
 	\sup_{t \in\,[0,T]}\Vert u_\e(t,\cdot)\Vert_0\leq \Vert g\Vert_0.	
 	\end{equation}}

 \end{Remark}

Now, we conclude this section with the proof of \eqref{109}.
\begin{Lemma}
If $X^{t,x}_{\e,m}$ is the solution of problem \eqref{108}, we have	
\begin{equation}
 \label{109-bis}
 \lim_{m\to\infty}\sup_{s \in\,[0,t]}\mathbb{E}\Vert X^{t,x}_{\e,m}(s)-X_\e^{t,x}(s) \Vert^2_H=0.	
 \end{equation}
\end{Lemma}

\begin{proof}
If we denote $\rho_{\e,m}(s):=X^{t,x}_{\e,m}(s)-X^{t,x}_\e(s)$ and $W_m(r)=W(r)-W^m(r)$, we have
\[\begin{array}{ll}
\ds{\rho_{\e,m}(s)=} & \ds{e^{sA}(J_mx-x)+\int_0^s e^{(s-r)A}\left(J_m b(X^{t,x}_{\e,m}(r))-b(X_\e^{t,x}(r))\right)\,dr}\\[18pt]
& \ds{+\int_0^s e^{(s-r)A}\left(J_m \Sigma_{\e,t}(r,X^{t,x}_{\e,m}(r))-\Sigma_{\e,t}(r,X_\e^{t,x}(r))\right)dW^m(r)}\\[18pt]
&\ds{+	\int_0^s e^{(s-r)A}\Sigma_{\e,t}(r,X_\e^{t,x}(r))dW_m(r).}\end{array}\]
Therefore, since $\Vert J_m\Vert_{\mathcal{L}(H)}\leq 1$, we have
\[\begin{array}{ll}
\ds{\mathbb{E}\Vert \rho_{\e,m}(s)\Vert^2_H\leq}  &  \ds{ c\Vert J_mx-x\Vert^2_H+c_t\int_0^s\mathbb{E}\Vert \rho_{\e,m}(r)\Vert^2_H\,dr}\\[18pt]
&\ds{+c_t\int_0^s\mathbb{E}\Vert J_m b(X_\e^{t,x}(r))-b(X_\e^{t,x}(r))\Vert_H^2\,dr}\\[18pt]
&\ds{+c\int_0^s\mathbb{E}\Vert e^{(s-r)A}J_m \left( \Sigma_{\e,t}(r,X^{t,x}_{\e,m}(r))-\Sigma_{\e,t}(r,X_\e^{t,x}(r))\right)\Vert_{\mathcal{L}_2(H)}^2\,dr	}\\[18pt]
&\ds{+c\int_0^s\mathbb{E}\Vert J_m e^{(s-r)A} \Sigma_{\e,t}(r,X_\e^{t,x}(r))-e^{(s-r)A}\Sigma_{\e,t}(r,X_\e^{t,x}(r))\Vert^2_{\mathcal{L}_2(H)}\,dr}\\[18pt]
&\ds{+c\int_0^s \mathbb{E} \Vert e^{(s-r)A}\Sigma_{\e,t}(r,X_\e^{t,x}(r))S_m\Vert^2_{\mathcal{L}_2(H)}\,dr=:\sum_{i=1}^6 I_{m,i}^\e(s),}
\end{array}\]
where $S_m:=I-P_m$ is the projection of $H$ onto $\text{span}\{e_{m+1}, e_{m+2},\ldots\}$.

By proceeding as in Section \ref{sec-spde}, we have that
\[I_{m,4}^\e(s)\leq c(\e)\int_0^s((s-r)\wedge 1)^{-\frac 12}\left(1+((t-r)\wedge 1)^{-2\varrho}\right)\mathbb{E}\Vert \rho_{\e,m}(r)\Vert_H^2\,dr,\]
so that 
\[\mathbb{E}\Vert \rho_{\e,m}(s)\Vert^2_H\leq c(\e)\int_0^s((s-r)\wedge 1)^{-(2\varrho+\frac 12)}\mathbb{E}\Vert \rho_{\e,m}(r)\Vert_H^2\,dr+\Lambda_{\e,m}(s),\]
where
\[\Lambda_{\e,m}(s):=I_{m,1}^\e(s)+I_{m,3}^\e(s)+I_{m,5}^\e(s)+I_{m,6}^\e(s).\]
Since $2\varrho+1/2<1$,  thanks to a generalized Gronwall's inequality (see \cite{Ye}), this implies that
\[\mathbb{E}\Vert \rho_{\e,m}(s)\Vert^2_H\leq c_{\e,t} \,\Lambda_{\e,m}(s)\leq c_{\e,t} \,\Lambda_{\e,m}(t),\\ \ \ \ s \in\,[0,t],\]
and \eqref{109-bis} follows if we can prove that
\begin{equation}
\label{120}
\lim_{m\to\infty} \Lambda_{\e,m}(t)=0. 	
\end{equation}

It is immediate to check that
\begin{equation}
\label{116}
\lim_{m\to\infty} I_{m,1}^\e+I_{m,3}^\e(t)=0.	
\end{equation}
Moreover, according to Hypothesis \ref{H3} and to the fact that $u$ is bounded in $[0,t]\times H$, for every $m \in\,\mathbb{N}$ we have
 \[\begin{array}{l}
\ds{\Vert J_m e^{(t-r)A} \Sigma_{\e,t}(r,X_\e^{t,x}(r))-e^{(t-r)A} \Sigma_{\e,t}(r,X_\e^{t,x}(r))\Vert_{\mathcal{L}_2(H)}}\\[18pt]
 \ds{\leq 2\,\Vert e^{(t-r)A} \sigma(X_\e^{t,x}(r),u_\e(t-r,X_\e^{t,x}(r))\Vert_{\mathcal{L}_2(H)}\leq c\,(t-r)^{-\frac 14}\left(1+\Vert X_\e^{t,x}(r)\Vert_H\right).}	
\end{array}
\]	
Then, 
since 
\[\lim_{m\to\infty} \Vert J_m e^{(t-r)A} \Sigma_{\e,t}(r,X_\e^{t,x}(r))-e^{(t-r)A} \Sigma_{\e,t}(r,X_\e^{t,x}(r))\Vert=0,\]
and since 
the mapping 
\[s \in\,[0,t]\mapsto (t-s)^{-\frac 14}\left(1+\Vert X_\e^{t,x}(s)\Vert_H\right) \in\,\mathbb{R},\] belongs to $L^2(\Omega;L^2([0,t]))$, by the dominated convergence theorem we have that 
\begin{equation}
\label{117}
\lim_{m\to\infty} I_{m,5}^\e(t)=0.	\end{equation}
In the same way, by the dominated convergence theorem we have also that
\begin{equation}
\label{118}
\lim_{m\to\infty} I_{m,6}^\e(t)=0.	\end{equation}
Therefore, combining together \eqref{116}, \eqref{117} and \eqref{118}, we obtain \eqref{120} and \eqref{109-bis} follows.
\end{proof}

\section{Existence and uniqueness of global classical solutions for the quasi-linear problem}

In Theorem \ref{lemma1} we have proved that  for every $\eta>1/2$ and $\vartheta \in\,(0,(\eta-1/2)\wedge 1)$,  there exist $\delta_1, T_1>0$ such that problem \eqref{quasi-linear-bis} has a mild solution $u_\e$ in $C_{\e,\varrho, \eta}((0,T_1];C^{2+\vartheta}_b(H))$. In Section \ref{ssec4.2} we have shown that such mild solution is in fact a classical solution. Our purpose here is first proving that $u_\e$ is defined on the \ interval $[0,T]$, for every $T>0$,  and then proving that it is the unique solution.

We start with the following a-priori bound.

\begin{Lemma}
There exists $\delta_2 \in\,(0,\delta_1]$, that depends only on $\Vert g\Vert_\eta$, such that if $u_\e$  is a mild solution of \eqref{quasi-linear-bis} for some $\delta\leq \delta_2$, belonging to $C_{\e,\varrho, \eta}((0,T];C^{2+\vartheta}_b(H))$, then 
\begin{equation}
\label{122}	
\Vert u\Vert_{\e, \varrho, \eta, \vartheta, T}\leq c_{\e}\,\Vert g\Vert_\eta,\ \ \ \ \ \e \in\,(0,1),
\end{equation}
for some constant $c_\e$ independent of $T>0$.
	
\end{Lemma}

\begin{proof}
In what follows, for any function $v:[0,T]\times H\to\mathbb{R}$ we define
\[N_\e(v(t)):=[ v(t,\cdot)]_\eta+\e^{\varrho}(t\wedge 1)^{\varrho} \Vert D v(t,\cdot)\Vert_\vartheta+\e^{\varrho+\frac 12}(t\wedge 1)^{\varrho+\frac 12}\Vert D^2v(t,\cdot)\Vert_\vartheta.	\]
With the notations  we have introduced in Section \ref{sec4}, thanks to \eqref{113} we have
\begin{equation}
\label{130}	
\Vert u_\e(t,\cdot)\Vert_0+N_\e(u_\e(t))\leq \Vert g\Vert_0+N_\e(R^\e_t g)+N_\e(\Gamma^\e_{\delta,1}(u_\e)(t))+N_\e(\Gamma^\e_{2}(u_\e)(t)),\end{equation}
where
\[\begin{array}{ll}
\ds{\Gamma^\e_{\delta,1}(u)(t,x)} & \ds{:=\int_0^tR^\e_{t-s}\,\gamma^\e_{\delta,1}(u,s)(x)\,ds:=\frac \e2\int_0^tR^\e_{t-s}\text{Tr}\left[\delta F(u)(s,\cdot)D^2u(s,\cdot)\right](x)\,ds}\end{array}
\]
and
\[\begin{array}{l}
\ds{\Gamma^\e_{2}(u)(t,x):=\int_0^tR^\e_{t-s}\gamma_2(u,s)(x)\,ds:=\int_0^tR^\e_{t-s}\langle b,D u(s,\cdot)\rangle_H(x)\,ds.}
\end{array}
\]
In \eqref{71-bis} we have already shown that
\begin{equation}
\label{131}
\sup_{t \in\,[0,T]}N_\e(R^\e_tg)\leq c\,\Vert g\Vert_\eta.	
\end{equation}
Thus, in order to prove \eqref{122} we need to estimate
$N_\e(\Gamma^\e_{\delta,1}(u_\e)(t))$ and $N_\e(\Gamma^\e_{2}(u_\e)(t))$. 

Thanks to \eqref{26} and \eqref{113}, we have
\[\Vert F(u_\e(s,\cdot))\Vert_0\leq c\left(1+\Vert u_\e(s,\cdot)\Vert_0\right)\leq c\left(1+\Vert g\Vert_0\right),\] 
and then
\begin{equation}
\label{123}
\begin{array}{l}
\ds{\Vert \gamma^\e_{\delta,1}(u_\e,s)\Vert_0\leq c\,\e\,\delta \Vert F(u_\e(s,\cdot))\Vert_0\Vert D^2 u_\e(s,\cdot)\Vert_0\leq c\,\e\,\delta\left(1+\Vert g\Vert_0\right)\Vert D^2 u_\e(s,\cdot)\Vert_0.}
\end{array}
	\end{equation}
Moreover, due to \eqref{27} and \eqref{113} we have
\[[F(u_\e(s,\cdot))]_\vartheta\leq c\left(1+\Vert u_\e(s,\cdot)\Vert_0+[u_\e(s,\cdot)]_\vartheta\right)\leq c\left(1+\Vert g\Vert_0+[u_\e(s,\cdot)]_\vartheta\right),\]
so that
\[\begin{array}{ll}
\ds{[\gamma^\e_{\delta,1}(u_\e,s)]_\vartheta}& \ds{\leq c\,\e\,\delta\,[ F(u_\e(s,\cdot))]_\vartheta\Vert D^2 u_\e(s,\cdot)\Vert_0+c\,\e\,\delta\Vert F(u_\e(s,\cdot))\Vert_0[D^2u_\e(s,\cdot)]_\vartheta}	\\[18pt]
&\ds{\leq c\,\e\,\delta\left(1+\Vert g\Vert_0\right)\Vert D^2 u_\e(s,\cdot)\Vert_\vartheta+ c\,\e\,\delta\,[u_\e(s,\cdot)]_\vartheta\Vert D^2 u_\e(s,\cdot)\Vert_0.}
\end{array}
	\]
According to \eqref{6} and \eqref{113}, this implies
\begin{equation}
\label{124}
\begin{array}{ll}
\ds{[\gamma^\e_{\delta,1}(u_\e,s)]_\vartheta} & \ds{\leq c\,\e\,\delta\left(1+\Vert g\Vert_0\right)\Vert D^2 u_\e(s,\cdot)\Vert_\vartheta+c\,\e\,\delta\,\Vert u_\e(s,\cdot)\Vert_0[ D^2 u_\e(s,\cdot)]_\vartheta}\\[18pt]
&\ds{\leq c\,\e\,\delta\left(1+\Vert g\Vert_0\right)\Vert D^2 u_\e(s,\cdot)\Vert_\vartheta.}
	\end{array}
\end{equation}
Therefore, if we combine together \eqref{123} and \eqref{124} we conclude that
\[
\begin{array}{ll}
\ds{\Vert \gamma^\e_{\delta,1}(u_\e,s)\Vert_\vartheta} & \ds{\leq c\,\e\,\delta\left(1+\Vert g\Vert_0\right)\Vert D^2 u_\e(s,\cdot)\Vert_\vartheta}\\[14pt]
&\ds{\leq c\,\e\,\delta\left(1+\Vert g\Vert_0\right)\e^{-(\varrho+\frac 12)}(s\wedge 1)^{-(\varrho+\frac 12)}\,N_\e(u_\e(s,\cdot)).}
\end{array}
	\]
By proceeding as in the proof of Lemma \ref{lemma4.4} (see also Remark \ref{remark5.5}), this allows to conclude
\begin{equation}
\label{132}
N_\e(\Gamma^\e_{\delta,1}(u_\e)(t))\leq c\,\delta\left(1+\Vert g\Vert_0\right)\sup_{s \in\,[0,t]}N_\e(u_\e(s,\cdot)),\ \ \ \ t \in\,[0,T].	
\end{equation}

Now, let us estimate $N_\e(\Gamma^\e_2(u_\e)(t))$.
We have
\[ \Vert\gamma_2(u_\e)(t)\Vert_{1}\leq c\,\left(\Vert Du_\e(t,\cdot)\Vert_0+\Vert D^2u_\e(t,\cdot)\Vert_0\right).\]
Thus, according to \eqref{4} and \eqref{7}, thanks to \eqref{113} for every $\a>0$ there exists $\kappa_\a>0$ such that

\begin{equation}
\label{141}
\Vert\gamma_2(u_\e)(t)\Vert_{1}\leq \a\,[D^2u_\e(s,\cdot)]_\vartheta+\kappa_\a \Vert u_\e(s,\cdot)\Vert_0.
\end{equation}
In view of \eqref{65} and \eqref{141}, there exists some $\lambda_\vartheta>0$ such that for every $t \in\,[0,T]$
\begin{equation}
\label{144}	
\begin{array}{l}
\ds{N_\e(\Gamma^\e_2(u)(t))\leq \a\int_0^t e^{-\la_\vartheta(t-s)}[D^2 u_\e(s,\cdot)]_\vartheta\,ds}\\[14pt]
\ds{+\a\,\e^{\varrho}\,(t\wedge 1)^\varrho \int_0^t e^{-\la_\vartheta(t-s)}\e^{-\frac \vartheta 2}((t-s)\wedge 1)^{-\frac\vartheta 2}[D^2 u_\e(s,\cdot)]_\vartheta\,ds}\\[14pt]
\ds{ +\a\,\e^{\varrho +\frac 12}(t\wedge 1)^{\varrho +\frac 12}\int_0^t e^{-\la_\vartheta(t-s)}\e^{-\frac{1+\vartheta} 2}((t-s)\wedge 1)^{-\frac{1+\vartheta} 2}[D^2 u_\e(s,\cdot)]_\vartheta\,ds}\\[14pt]
\ds{+\kappa_\a\int_0^t e^{-\la_\vartheta(t-s)}\e^{-\frac{1+\vartheta} 2}((t-s)\wedge 1)^{-\frac{1+\vartheta} 2}\,ds\Vert g\Vert_0}\\[14pt]
\ds{\leq c\,\a \e^{-\frac{1+\vartheta}2}\sup_{s \in\,[0,t]}N_\e(u_\e(s,\cdot))+c\,\kappa_\a \e^{-\frac{1+\vartheta}2}\,\Vert g\Vert_0.}
\end{array}
\end{equation}
 Hence,  if we plug \eqref{131}, \eqref{132} and \eqref{144} into \eqref{130}, we obtain
\[\begin{array}{l}
\ds{	\Vert u_\e(t,\cdot)\Vert_\eta+N_\e(u_\e(t,\cdot))}\\[14pt]
\ds{\leq c\,\Vert g\Vert_\eta+c\left[\delta\left(1+\Vert g\Vert_0\right)+\a\e^{-\frac{1+\vartheta}2}\right]\sup_{s \in\,[0,t]}N_\e(u_\e(s,\cdot))+c\,\kappa_\a \e^{-\frac{1+\vartheta}2}\,\Vert g\Vert_0.}
\end{array}\]
In particular if take  $\delta_2\leq \delta_1$ such that 
\[c\,\delta_2\left(2+\Vert g\Vert_0\right)<1/2,\] and $\a=\delta \e^{\frac{1+\vartheta}2}$, we obtain \eqref{122} for every $\delta\leq \delta_2$.

\end{proof}

\subsection{Conclusion of the proof of Theorem \ref{teo1}}

Thanks to \eqref{122}, by standard arguments we have that for every $\e \in\,(0,1)$ the local solution we found in Theorem \ref{lemma1} is in fact a global solution.
Moreover, this global solution is  unique. Actually, if $u_1, u_2  \in\,C_{\e,\varrho, \eta}((0,T];C^{2+\vartheta}_b(H))$ are two solutions of equation \eqref{quasi-linear-tris}, for some fixed $\delta\leq \delta_2$, we assume that 
\[t_0:=\sup\left\{t \in\,(0,T]\ :\ u_1(s)=u_2(s),\ s \in\,[0,t]\right\}<T.\]
With the same notations we have used in Section \ref{sec4}, we introduce the problem
\begin{equation}
\label{150}
u(t)=\Gamma_{\e,\delta}^{\varphi}(u)(t)=R^\e_t\varphi+\Gamma_{\e,\delta}(u)(t),\ \ \ \ \ t\geq t_0,	
\end{equation}
where $\varphi:=u_1(t_0)=u_2(t_0)$. 
Due to \eqref{122}, we have that 
\[\Vert \varphi\Vert_\eta\leq c_{\e,\delta}\,\Vert g\Vert_\eta,\] 
for some constant $c_{\e,\delta}>0$ independent of $T>0$.

As shown in Section \ref{sec4}, there exist $\bar{R}$, $\bar{\tau}>0$ and $\bar{\delta}\leq \delta_2$ such that the mapping $\Gamma_{\e,\delta}^{\varphi}$ maps $\mathcal{Y}^{\e,\bar{R}}_{\varrho,\eta,\vartheta,t_0,\bar{\tau}}$ into itself as a contraction, for every $\delta\leq \bar{\delta}$, where
\[\mathcal{Y}^{\e,R}_{\varrho,\eta,\vartheta,t_0,\bar{\tau}}:=\left\{\,u \in\,C_{\e,\varrho,\eta}((t_0,t_0+\bar{\tau}];C^{2+\vartheta}_b(H))\ :\ \Vert u\Vert_{\e, \varrho, \eta,\vartheta, t_0,\bar{\tau}}\leq \bar{R}\right\},\]
and $C_{\e, \varrho, \eta}((t_0,t_0+\bar{\tau}];C^{2+\vartheta}_b(H))$ is the space of all functions $ u$ belonging to  $C([t_0,t_0+\bar{\tau}];C^\eta_b(H))\cap C((t_0,t_0+\bar{\tau}];C^{2+\vartheta}_b(H))$ such that the norm
\[\begin{array}{l}
\ds{\Vert u\Vert_{\e,\varrho,\eta, \vartheta, t_0,\bar{\tau}}}\\[14pt]
\ds{:=\sup_{t \in\,(t_0,t_0+\bar{\tau}]}\left(\Vert u(t,\cdot)\Vert_\eta+\e^{\varrho}((t-t_0)\wedge 1)^{\varrho}\Vert D_xu(t,\cdot)\Vert_{\vartheta}+\e^{\varrho+\frac 12}((t-t_0)\wedge 1)^{\varrho+\frac 12}\Vert D^2_xu(t,\cdot)\Vert_\vartheta\right)	}
\end{array}\]
is finite.

 In particular  $\Gamma_{\e,\delta}^{\varphi}$ has a unique fixed point in $\mathcal{Y}^{\e,\bar{R}}_{\varrho,\eta,\vartheta,t_0,\bar{\tau}}$ or, equivalently, equation \eqref{150} has a unique solution on the interval $[t_0,t_0+\bar{\tau}]$. This implies that 
 \[u_1(s)=u_2(s),\ \ \ \ \ s \in\,[0,t_0+\bar{\tau}],\]
 violating the definition of $[0,t_0]$ as the maximal interval where $u_1$ and $u_2$ coincide.

\section{The large deviation principle}
\label{sec8}
In this last section we give a proof of Theorem \ref{teo3}. We follow the well-known method based on weak convergence, as developed in \cite{bdm}. To this purpose, we need to introduce some notations.

For every $t>0$, we denote by $\mathcal{P}_t$ the set of predictable processes in $L^2(\Omega\times [0,t];H)$, and for every $M>0$ we introduce the sets
\[\mathcal{S}_{t,M}:=\left\{ \varphi \in\,L^2_w(0,t;H)\ :\ \Vert \varphi\Vert_{L^2(0,t;H)}\leq M\right\},\]
and
\[\Lambda_{t,M}:=\left\{ \varphi \in\,\mathcal{P}_t\ :\ \varphi \in \mathcal{S}_{t,M},\ \mathbb{P}-\text{a.s.}\right\}.\]

In Theorem \ref{teo2} we have shown that  for  every  $M, t>0$ and $\varphi \in\,\Lambda_{t,M}$ and for every $x \in\,H$ and $\e \in\,(0,1)$ there exists a unique mild solution $X^{t, x}_{\varphi,\e} \in\,L^2(\Omega;C([0,t];H))$ for equation \eqref{stoch-pde}.

Next, we consider the  problem 
\begin{equation}
\label{sbm2-bis}
\frac {dX}{ds}(s)=A	X(s)+b(X(s))+\sigma(X(s),g(Z^{X(s)}(t-s)))\varphi(s),\ \ \ \ X(0)=x,	\end{equation}
where, as we did in Section \ref{sec3}, for every $y \in\,H$ we denote by $Z^y$ the solution of equation \eqref{sbm5}.
In what follows, we show that the following result holds.
\begin{Proposition}
\label{prop1}
Assume that $g:H\to\mathbb{R}$ is Lipschitz-continuous. Then, under the same assumptions of Theorem \ref{teo2}, for  every  $t>0$ and $\varphi \in\,L^2(0,t;H)$ and for every $x \in\,H$, there exists a unique mild solution $X^{t, x}_{\varphi} \in\,C([0,t];H)$ for equation \eqref{sbm2-bis}.	
\end{Proposition}

Once proved Theorem \ref{teo2} and Proposition \ref{prop1}, 
we introduce the following two conditions.
\begin{enumerate}

\item[C1.] Let $\{\varphi_\epsilon\}_{\epsilon>0}$ be an arbitrary family of processes in $ \Lambda_{t,M}$ such that 
\[\lim_{\e\to 0} \varphi_\e=\varphi,\ \ \ \ \text{in distribution in}\ \ \ L_w^2(0,t;H),\]
where $L^2_w(0,t;H)$ is the space $L^2(0,t;H)$ endowed with the weak topology and  $\varphi \in\,\Lambda_{t,M}$. Then we have
\[\lim_{\e\to 0} X^{t,x}_{\varphi_\e,\e}=X^{t,x}_\varphi,\ \ \ \ \text{in distribution}\ \ \ C([0,t],H).\]
\item[C2.] For every $t, R>0$,  the level sets $\Phi_{t,R}=\{I_{t,x}	\leq R\}$ are compact in the space $C([0,t];H)$.

\end{enumerate}

As shown in \cite{bdm}, 
Conditions  C1. and C2. imply that the  family $\{X^{t,x}_\epsilon\}_{\epsilon \in\,(0,1)}$ satisfies a Laplace principle with action functional $I_{t,x}$ in the space $C([0,t];H)$ for the . Due to the compactness of the level sets $\Phi_{t, R}$ stated in C2. this is equivalent to the validity of Theorem \ref{teo3}.

\subsection{Proof of Proposition \ref{prop1}}
For every $y \in\,H$ and $s \in\,[0,t]$  we define
\[\Sigma_t(y,s):=\sigma(y,g(Z^y(t-s))).\]
With this notation, a function in $C([0,t];H)$ is a mild solution for equation \eqref{sbm2-bis} if it is a fixed point of the mapping  $\Lambda_t$ defined for every $X \in\,C([0,t];H)$ by
\[\Lambda_t(X)(s):=e^{sA}x+\int_0^s e^{(s-r)A}b(X(r))\,dr+\int_0^s  e^{(s-r)A}\Sigma_t(X(r),r)\,\varphi(r)\,dr,\ \ \ \ s \in\,[0,t].\]  It is immediate to check that 
there exists a continuous increasing function $\kappa(s)$ such that for every $y_1, y_2 \in\,H$
\begin{equation} \label{sbm42}
\Vert Z^{y_1}(s)-Z^{y_2}(s)\Vert_H\leq \kappa(s)\,\Vert y_1-y_2\Vert_H,\ \ \ \ \ \ s\geq 0.	
\end{equation}
Hence,  since we are assuming that $g:H\to\mathbb{R}$ is Lipschitz-continuous, according to Hypothesis \ref{H1} for every $y_1, y_2, h \in\,H$ we have
\[\Vert [\Sigma_t(y_1,r)-\Sigma_t(y_2,r)]h\Vert_H\leq c\,\left(1+\kappa(t-r)\right)\Vert y_1-y_2\Vert_H\Vert h\Vert_H,\ \ \ \ \ r \in\,[0,t].\]
In particular, for every $X_1, X_2 \in\,C([0,t];H)$ and $s \in\,[0,t]$ we have
\[\begin{array}{ll}
\ds{\Vert \Lambda_t(X_1)(s)-\Lambda_t(X_2)(s)\Vert_H} & \ds{\leq c\,\int_0^s \left(1+\left(1+\kappa(t-r)\right)\,\Vert\varphi(r)\Vert_H\right)\Vert X_1(r)-X_2(r)\Vert_H\,dr}\\[14pt]
&\ds{\leq c_t\left(\Vert \varphi\Vert_{L^2(0,t;H)}+1\right)\Vert X_1-X_2\Vert_{C([0,t];H)}.}	
\end{array}\]
This implies that $\Lambda_t:C([0,t];H)\to C([0,t];H)$ is Lipschitz continuous and by standard arguments we conclude that $\Lambda_t$ has a unique fixed point. 

\subsection{Proof of the validity of Conditions C1 and C2}
It is enough to prove Condition C1. Actually, from the proof of Condition C1 we will see that the mapping
\[\varphi \in\,L^2_w(0,t;H)\mapsto Y^{t,x}_\varphi\]
is continuous and 
\[\Phi_{t,R}=\{I_{t,x}\leq R\}=\{Y^{t,x}_\varphi\ :\ \varphi \in\,\Lambda_{t,c(R)}\},\] for some $c(R)>0$. Therefore,  since the set $\Lambda_{t,M}$ is compact in $L^2_w(0,t;H)$ for every $M>0$,   we conclude that Condition C2 holds.

 In order to prove Condition C1, we first need to prove the following preliminary results
 \begin{Lemma}
 \label{lem1}
 Under the same assumptions of Theorem \ref{teo3}, for every $p\geq 1$ we have
	\begin{equation}
	\label{sbm11}
	\sup_{\e \in\,(0,1)}\mathbb{E}\sup_{s \in\,[0,t]}\Vert X^{t, x}_{\varphi_\e, \e}(s)\Vert_H^p\leq c(t,M,p)\left(1+\Vert x\Vert^p_H\right).	
	\end{equation}
	\end{Lemma}
\begin{proof}
We have
\[\begin{array}{ll}	
\ds{X^{t,x}_{\varphi_\e,\e}(s):=} & \ds{e^{sA}x+\int_0^s e^{(s-r)A}b(X^{t,x}_{\varphi_\e,\e}(r))\,ds+\int_0^s e^{(s-r)A}\Sigma_{t,\e}(r,X^{t,x}_{\varphi_\e,\e}(r))\,\varphi_\e(r)\,dr}\\[14pt]
&\ds{+\sqrt{\e}\int_0^s e^{(s-r)A}\Sigma_{t,\e}(r,X^{t,x}_{\varphi_\e,\e}(r))\,dW(r),}
\end{array}
\]
where $\Sigma_{t,\e}$ is the operator defined in \eqref{sigma}. Hence, for every $s \in\,[0,t]$ and $p\geq 1$ we have
\begin{equation}
\label{sbm15}\begin{array}{ll}
\ds{\Vert X^{t,x}_{\varphi_\e,\e}(s)\Vert_H^p\leq}  &  \ds{ c_p\,\Vert x\Vert_H^p+c_p\int_0^s \Vert X^{t,x}_{\varphi_\e,\e}(r)\Vert_H^p	\,dr+c_{p,M}\left(\int_0^s \Vert e^{(s-r)A}\Sigma_{t,\e}(r,X^{t,x}_{\varphi_\e,\e}(r))\Vert_{\mathcal{L}(H)}^2\,dr\right)^{\frac p2}}\\[14pt]
&\ds{+c_p\left\Vert \int_0^s e^{(s-r)A}\Sigma_{t,\e}(r,X^{t,x}_{\varphi_\e,\e}(r))\,dW(r)\right\Vert_H^p+c_{p,t}.}
\end{array}
	\end{equation}
According to Hypothesis \ref{H3}, for every $\tau>0$, $s \in\,[0,t]$  and $x \in\,H$ we have
\[\Vert e^{\tau A} \Sigma_{t,\e}(s,x)\Vert^2_{\mathcal{L}_2(H)}\leq c\,(\tau\wedge 1)^{-\frac 12}\left(\Vert x\Vert_H^2+|u_\e(t-s,x)|^2+1\right).\]
Moreover, according to \eqref{106}, we have
\[\sup_{(s,x) \in\,[0,t]\times H}|u_\e(s,x)|\leq \Vert g\Vert_0,\ \ \ \ \ \e \in\,(0,1),\]
so that
\begin{equation}
\label{sbm16}	
\sup_{\e \in\,(0,1)}\Vert e^{\tau A} \Sigma_{t,\e}(s,x)\Vert^2_{\mathcal{L}_2(H)}\leq c\,(\tau\wedge 1)^{-\frac 12}\left(\Vert x\Vert_H^2+1\right).
\end{equation}
In particular,
\begin{equation}
\label{sbm30}	
\mathbb{E}\sup_{r \in\,[0,s]}\left(\int_0^s \Vert e^{(s-r)A}\Sigma_{t,\e}(r,X^{t,x}_{\varphi_\e,\e}(r))\Vert_{\mathcal{L}(H)}^2\,dr\right)^{\frac p2}\leq c_{p,t}\,\left(\mathbb{E}\sup_{r \in\,[0,s]}\Vert X^{t,x}_{\varphi_\e,\e}(r)\Vert_H^p+1\right).
\end{equation}

Now, if we fix $p>4$, we can find $\a<1/4$ such that $(\a-1)p/(p-1)>-1$. By using a stochastic factorization argument, we have
\[\begin{array}{l}
\ds{\int_0^s e^{(s-r)A}\Sigma_{t,\e}(r,X^{t,x}_{\varphi_\e,\e}(r))\,dW(r)=c_\a\int_0^s e^{(s-r)A}(s-r)^{\a-1}Y_{\a,\e}(r)\,dr,}	
\end{array}\]
where
\[Y_{\a,\e}(r):=\int_0^re^{(r-\rho)A}(r-\rho)^{-\a}\Sigma_{t,\e}(\rho,X^{t,x}_{\varphi_\e,\e}(\rho))\,dW(\rho).\]
Then,  we obtain
\[\begin{array}{l}
\ds{\left\Vert\int_0^s e^{(s-r)A}\Sigma_{t,\e}(r,X^{t,x}_{\varphi_\e,\e}(r))\,dW(r)\right\Vert^p_H\leq c_{\a,p}\left(\int_0^s (s-r)^{\frac{(\a-1)p}{p-1}}\,dr\right)^{p-1}\int_0^s\Vert Y_{\a,\e}(r)\Vert_H^p\,dr,}
\end{array}\]
so that, thanks to \eqref{sbm16} and to the fact that $\a<1/4$
\begin{equation}
\label{sbm26}
\begin{array}{l}
\ds{\mathbb{E}\sup_{r \in\,[0,s]}\left\Vert\int_0^s e^{(s-r)A}\Sigma_{t,\e}(r,X^{t,x}_{\varphi_\e,\e}(r))\,dW(r)\right\Vert^p_H\leq c_{\a,p,t}\int_0^s\mathbb{E}\Vert Y_{\a,\e}(r)\Vert_H^p\,dr}	\\[14pt]
\ds{\leq c_{\a,p,t}\int_0^s\left(\mathbb{E}\int_0^r (r-\rho)^{-(\frac 12+2\a)}\left(\Vert X^{t,x}_{\varphi_\e,\e}(\rho)\Vert_H^2+1\right)\,d\rho\right)^{\frac p2}\,dr}\\[14pt]
\ds{\leq c_{\a,p,t}\left(\int_0^s \mathbb{E}\sup_{\rho \in\,[0,r]}\Vert X^{t,x}_{\varphi_\e,\e}(\rho)\Vert_H^p\,dr+1\right).}
\end{array}\end{equation}

Therefore, thanks to \eqref{sbm15}, \eqref{sbm30} and \eqref{sbm26}, 
\[\begin{array}{ll}
\ds{\mathbb{E}\sup_{r \in\,[0,s]}\Vert X^{t,x}_{\varphi_\e,\e}(r)\Vert_H^p\leq}  &  \ds{ c_{t,M,p}\left(\Vert x\Vert_H^p+1\right)+c_{p,t}\int_0^s \mathbb{E}\sup_{\rho \in\,[0,r]}\Vert X^{t,x}_{\varphi_\e,\e}(\rho)\Vert_H^p\,dr,}
	\end{array}\]
	and  Gronwall's Lemma allows to conclude.

\end{proof}

\begin{Lemma}
\label{lem2}
 Under the same assumptions of Theorem \ref{teo3}, we have
 \begin{equation}
 \label{sbm40}
 |u_\e(s,x)-g(Z^x(s))|	\leq c_t\,\sqrt{\e}\left(1+\Vert x\Vert_H\right),\ \ \ \ s \in\,[0,t].
 \end{equation}

\end{Lemma}
\begin{proof}
Thanks to \eqref{106}, we have
\[u_\e(s,x)-g(Z^x(s))=\mathbb{E}\left(g(X^{s,x}_\e(s))-g(Z^x(s))\right),\]
so that, since we are assuming that $g$ is Lipschitz-continuous
\[|u_\e(s,x)-g(Z^x(s))|\leq c\,\mathbb{E}\,\Vert X^{s,x}_\e(s)-Z^x(s)\Vert_H.\]
Now, if we define $\rho_\e^x(s):=X^{s,x}_\e(s)-Z^x(s)$, we have
\[\begin{array}{ll}
\ds{\rho_\e^x(s)=} & \ds{\int_0^s e^{(s-r)A}\left(b(X^{s,x}_\e(r))-b(Z^x(r))\right)\,dr+\sqrt{\e}\int_0^s e^{(s-r)A}\Sigma_{s,\e}(r,X^{s,x}_\e(r))\,dW(r),}\end{array}\]
where $\Sigma_{s,\e}$ is the operator introduced in \eqref{sigma}.
Due to \eqref{sbm11} and \eqref{sbm26}, we have
\[\mathbb{E}\Vert \rho_\e(s)\Vert_H\leq c\int_0^s \mathbb{E}\Vert \rho_\e(r)\Vert_H\,dr+c_t\,\sqrt{\e}\left(1+\vert x\Vert_H\right),\]
and Gronwall's lemma allows to conclude.
	
\end{proof}
	
Now, we are ready to prove condition C1. Let  $\{\varphi_\epsilon\}_{\epsilon>0}$ be an arbitrary family of processes in $ \Lambda_{t,M}$ converging in distribution, with respect to the weak topology of $L^2(0,t;H)$, to some 
  $\varphi \in\,\Lambda_{t,M}$.  As a consequence of  Skorohod theorem, we can assume that the sequence $\{\varphi_\epsilon\}_{\epsilon>0}$  converges $\mathbb{P}$-a.s. to $\varphi$, with respect to the weak topology of  $L^2(0,t;H)$. We will prove that this implies that 
  \begin{equation}
  \label{sbm45}
  \lim_{\e \to 0}	\sup_{s \in\,[0,t]}\mathbb{E}\Vert X_{\varphi_\e, \e}^{t,x}(s)-X^{t,x}_\varphi(s)\Vert_H^2=0.
  \end{equation}
If we define
	\[\rho_\e(s):=X_{\varphi_\e, \e}^{t,x}(s)-X^{t,x}_\varphi(s),\ \ \ \ s \in\,[0,t],\]
we have
 \begin{equation}
 \label{sbm20}
 \begin{array}{l}
\ds{	\rho_\e(s)=\int_0^s e^{(s-r)A}\left[b(X_{\varphi_\e, \e}^{t,x}(r))-b(X_{\varphi}^{t,x}(r))\right]	\,dr}\\[14pt]
\ds{+\int_0^se^{(s-r)A}\left[\sigma(X_{\varphi_\e, \e}^{t,x}(r),u_\e(t-r,X_{\varphi_\e, \e}^{t,x}(r)))\,\varphi_\e(r)-\sigma(X_{\varphi}^{t,x}(r),g(Z^{X^{t,x}_\varphi(r)}(t-r)))\,\varphi(r)\right]\,dr}\\[14pt]
\ds{+\sqrt{\e}\,\int_0^se^{(s-r)A}\Sigma_{t,\e}(r,X_{\varphi, \e}^{t,x}(r))\,dW_r=:\sum_{k=1}^3I_{k,\e}(s).}
\end{array}	
 \end{equation}

For $I_{1,\e}(s)$, due to the Lipschitz continuity of $b$, we have
\begin{equation}
\label{sbm21}
\Vert I_{1,\e}(s)\Vert^2_H\leq c_t\int_0^s \Vert \rho_\e(s)\Vert^2_H\,ds.	
\end{equation}

Concerning $I_{2,\e}(s)$, it can be written as
\[\begin{array}{l}
\ds{\int_0^se^{(s-r)A}\left[\sigma(X_{\varphi_\e, \e}^{t,x}(r),u_\e(t-r,X_{\varphi_\e, \e}^{t,x}(r)))-\sigma(X_{\varphi}^{t,x}(r),g(Z^{X_{\varphi_\e, \e}^{t,x}(r)}(t-r)))\right]\,\varphi_\e(r)\,dr}	\\[14pt]
\ds{+\int_0^se^{(s-r)A}\left[\sigma(X_{\varphi}^{t,x}(r),g(Z^{X_{\varphi_\e, \e}^{t,x}(r)}(t-r)))-\sigma(X_{\varphi}^{t,x}(r),g(Z^{X^{t,x}_\varphi(r)}(t-r)))\right]\,\varphi_\e(r)\,dr}\\[14pt]
\ds{+\int_0^se^{(s-r)A}\sigma(X_{\varphi}^{t,x}(r),g(Z^{X^{t,x}_\varphi(r)}(t-r)))\,(\varphi_\e(r)-\varphi(r))\,dr=:\sum_{k=1}^3J_{k,\e}(s).}
\end{array}\]
According to \eqref{sbm40}, we have
\[\Vert J_{1,\e}(s)\Vert_H\leq c\int_0^s\left(\Vert \rho_\e(r)\Vert_H+c_t\,\sqrt{\e} \left(1+\Vert X_{\varphi_\e, \e}^{t,x}(r)\Vert_H\right)\right)\Vert\varphi_\e(r)\Vert_H\,dr,\]
so that
\begin{equation}
\label{sbm41}\Vert J_{1,\e}(s)\Vert^2_H\leq c_{t,M}\int_0^s \Vert \rho_\e(r)\Vert^2_H\,dr+\e\,c_{t,M}\le(1+\sup_{r \in\,[0,t]}\Vert X_{\varphi_\e, \e}^{t,x}\Vert^2_H\right).\end{equation}
Moreover, thanks to \eqref{sbm42}, we have
\begin{equation}
\label{sbm43}
\Vert J_{2,\e}(s)\Vert^2_H\leq c_{t,M}\int_0^s \Vert Z^{X_{\varphi_\e, \e}^{t,x}(r)}(t-r)-Z^{X^{t,x}_\varphi(r)}(t-r)\Vert^2_H\,dr\leq c_{t,M}\int_0^s \Vert\rho_\e(r)\Vert^2_H\,dr.\end{equation}
Finally, for $I_{3,\e}(s)$, thanks to \eqref{sbm11} and \eqref{sbm26}  we have
\begin{equation}
\label{sbm22}
\mathbb{E}\sup_{s \in\,[0,t]}\Vert I_{3,\e}(s)\Vert^2_H\leq c_t\e\left(1+\Vert x\Vert^2_H\right).	
\end{equation}

Therefore, if we plug \eqref{sbm21}, \eqref{sbm41}, \eqref{sbm43} and \eqref{sbm22} into \eqref{sbm20}, in view of \eqref{sbm11} we obtain
\[\mathbb{E}\,\Vert \rho_\e(s)\Vert_H^2\leq c_{t,M}\int_0^s \mathbb{E}\Vert \rho_\e(r)\Vert_H^2\,dr+c_{t,M}\,\e\left(1+\Vert x\Vert_H^2\right)+\mathbb{E}\Vert J_{3,\e}(s)\Vert^2_H,\]
and the Gronwall lemma gives
\begin{equation}
\label{sbm46}
\mathbb{E}\,\Vert \rho_\e(s)\Vert_H^2\leq c_{t,M}\,\e\left(1+\Vert x\Vert_H^2\right)+c_{t,M}\int_0^s\mathbb{E}\Vert J_{3,\e}(r)\Vert^2_H\,dr.	\end{equation}
Now, due to the $\mathbb{P}$-a.s. convergence of $\varphi_\e$ to $\varphi$ in $L^2_w(0,t;H)$, since $\{\varphi\}_{\e \in\,(0,1)}\subset \Lambda_{t,M}$ and $\varphi \in\,\Lambda_{t,M}$ we can apply the dominated convergence theorem and we get
\[\lim_{\e\to 0}\int_0^s\mathbb{E}\Vert J_{3,\e}(r)\Vert^2_H\,dr.\]
Therefore, by taking the limit as $\e$ goes to zero in both sides of \eqref{sbm46} we obtain \eqref{sbm45}.

\end{document}